\documentclass[final,1p,times,nopreprintline]{elsarticle}

\usepackage{amssymb}
\usepackage{amsthm}
\usepackage{hyperref}
\usepackage{amsmath,amssymb,amsopn,amsfonts,mathrsfs,amsbsy,amscd}
\usepackage{longtable}
\usepackage{color}
\usepackage{mathtools}
\usepackage{ulem}

\newtheorem{theorem}{Theorem}[section]
\newtheorem{proposition}{Proposition}[section]
\newtheorem{lemma}{Lemma}[section]
\newtheorem{corollary}{Corollary}[section]
\theoremstyle{definition}
\newtheorem{example}{Example}[section]

\newtheorem{remark}{Remark}[section]
\newtheorem{definition}{Definition}[section]
\numberwithin{equation}{section}

\newcommand{\ro}{\circ}

\newcommand{\somkn}{\displaystyle\sum_{k=1}^{n}}
\newcommand{\somzpk}{\displaystyle\sum_{p=0}^{k}}
\newcommand{\somzpdqk}{\displaystyle\sum_{0\leq p<q\leq k}}

\newcommand{\rel}{\mathbb{R}}
\newcommand{\cplx}{\mathbb{C}}
\newcommand{\nat}{\mathbb{N}}

\newcommand{\w}{\omega}
\newcommand{\et}{\quad \textnormal{and}\quad}
\newcommand{\id}{\operatorname{Id}}

\newcommand{\tr}{\operatorname{tr}}

\newcommand{\g}{\mathfrak{g}}
\newcommand{\ai}{\mathfrak{a}}
\newcommand{\h}{\mathcal{H}}

\newcommand{\ad}{\operatorname{ad}}
\newcommand{\Span}{\operatorname{span}}

\newcommand{\lh}{\operatorname{L}^\h}
\newcommand{\lk}{\operatorname{L}^\kl}

\newcommand{\LS}{\operatorname{L}}

\newcommand{\J}{\operatorname{J}}
\numberwithin{equation}{section}

\newcommand{\aj}{\divideontimes}

\newcommand{\z}{\mathfrak{z}}
\newcommand{\ii}{\mathfrak{i}}
\newcommand{\D}{\operatorname{D}}
\newcommand{\ef}{\operatorname{F}}

\newcommand{\gl}{\mathfrak{gl}}
\newcommand{\cro}{[\,\cdot\,,\cdot\,]}

\newcommand{\s}{\mathfrak{s}}

\newcommand{\jj}{\mathfrak{j}}

\newcommand{\bi}{\mathcal{B}}
\newcommand{\kl}{\mathcal{B}}
\date{}

\begin{document}
\begin{frontmatter}
\title{A complete description of solvable symplectic Lie algebras}

\author{
	Abdelhak Abouqateb\footnote{Faculty of Sciences and Technologies, Cadi Ayyad University, Laboratoire MAIST,  B.P. 549 Gueliz Marrakesh, Morocco. a.abouqateb@uca.ac.ma}, Sa\"{\i}d Benayadi\footnote{ Universit\'e de Lorraine, Laboratoire IECL, CNRS-UMR 7502, UFR MIM, 3 rue Augustin Frenel, BP 45112, 57073 Metz Cedex 03, France. said.benayadi@univ-lorraine.fr }, Othmane Dani\footnote{Faculty of Sciences and Technologies, Cadi Ayyad University, Laboratoire MAIST, B.P. 549 Gueliz Marrakesh, Morocco. {othmanedani}@gmail.com   }
}
\begin{abstract}
In this paper, we present a complete characterization of solvable symplectic Lie algebras via a symplectic double extension process. We demonstrate that any such algebra is either symplectically irreducible or can be constructed through a finite sequence of symplectic double extensions by a line or a plane, starting from symplectically irreducible Lie algebras. Furthermore, we show that if a symplectic Lie algebra has a nondegenerate derived ideal, then it is necessarily unimodular and, in particular, solvable. Finally, we present a novel algebraic proof of a classical structural theorem on symplectically irreducible symplectic Lie algebras and classify all Lie algebras of dimension up to $6$ that admit such structures.
 
\end{abstract}
%\begin{keyword}
%Symplectic structures - Solvable, nilpotent algebras - Lie algebras of Lie groups
%{\it{ 2020 Mathematics Subject Classification:}} 
%53Dxx, 17B30, 22E60
%\end{keyword}
\end{frontmatter}

\section{Introduction}
A \textit{symplectic Lie algebra} $(\g,\cro,\omega)$ is a Lie algebra $(\g,\cro)$ endowed with a nondegenerate $2$-cocycle $\w\in\mathcal{Z}^2(\g,\rel)$, i.e., a nondegenerate, skew-symmetric bilinear form $\omega:\mathfrak{g}\times\mathfrak{g}\to{\mathbb{R}}$ such that:    
\begin{equation*}
	\w([x,y],z)+\w([z,x],y)+\w([y,z],x)=0,\qquad\forall\,x,y,z\in\g.
\end{equation*}
This condition ensures compatibility between the Lie bracket and the symplectic structure, linking algebra to geometry in a profound way. A fundamental result states that symplectic Lie algebras are in one-to-one correspondence with \textit{simply connected symplectic Lie groups} $(G, \widetilde{\omega})$. Here, $G$ is a Lie group carrying a left-invariant symplectic form $\widetilde{\omega}$, that is to say, $\widetilde{\omega}$ is a nondegenerate, closed $2$-form that is preserved under left translations. These structures sit at the crossroads of symplectic geometry and Lie theory, offering deep insights into classical mechanics, Poisson geometry, and even complex geometry. Due to their rich structure and broad applicability, symplectic Lie groups and algebras remain a vibrant area of research, as seen in works such as \cite{Bau, BeGo} and related literature.

An open problem in the theory of symplectic Lie algebras is their classification, which appears to be challenging to achieve, even in low dimensions, and whose construction has been explored through various approaches, including extensions and decompositions (see \cite{AM,BaBe,Bau,DM,Fi,Ov}). For example, the symplectic double extension process, introduced in \cite{MR} and later corrected and completed in \cite{DM}, has proven to be a powerful tool for constructing symplectic Lie algebras from simpler components. Moreover, in some cases, it provides inductive descriptions of certain classes of symplectic Lie algebras, unlike the quadratic double extension process (cf. \cite{MR-Q}) which provides an inductive description of all quadratic Lie algebras, i.e., Lie algebras of Lie groups with bi-invariant pseudo-Riemannian structures. The difficulty of applying this process in the symplectic setting arises from the crucial requirement that the symplectic Lie algebra contains a \textit{normal} isotropic ideal (i.e., an isotropic ideal whose orthogonal with respect to the symplectic structure is also an ideal). For instance, if the center of the Lie algebra is nonzero (e.g., when $(\g,\cro)$ is nilpotent), the symplectic double extension process proceeds smoothly without any issue. However, if the symplectic Lie algebra $(\g,\cro,\w)$ has an isotropic ideal that is not necessarily normal, the methods introduced in \cite{MR} are no longer applicable. For this reason, the authors of \cite{DM} propose a generalization of this process to cases where an isotropic ideal exists, though not necessarily normal. This generalization turns out to be not particularly useful for describing symplectic Lie algebras when the dimension of the isotropic ideal is greater than or equal to two. Moreover, the proposed approach relies on the left-symmetric product associated with the symplectic structure, making it difficult to apply even when the isotropic ideal has dimension one or two. This is why the only inductive description provided in \cite{DM} concerns nilpotent symplectic Lie algebras and relies on isotropic central ideals of dimension one, which are of course normal. So, the goal of this paper is twofold. First, it aims to extend the symplectic double extension process in a clearer and more useful manner when the isotropic ideal has dimension one or two. We will show that this is sufficient to obtain a complete description of symplectic solvable Lie algebras. Specifically, we will show that such a Lie algebra $(\g,\cro,\w)$ can be constructed as a symplectic double extension of another solvable symplectic Lie algebra $(\ai,\cro_\ai,\w_\ai)$ by either a line $\rel$ or a plane $\rel^2$ (see Sections \ref{Ar4sect1}, \ref{Section3}, and \ref{Section3p}). This approach will provide an inductive framework for describing all solvable symplectic Lie algebras, starting with the symplectically irreducible ones classified in \cite{Bau}. Furthermore, we apply this symplectic double extension process to unimodular symplectic Lie algebras to obtain a similar description. Second, the paper seeks to provide a new completely algebraic proof of a structural theorem regarding symplectically irreducible symplectic Lie algebras (see Section \ref{Art4SectionIrreducible}).\\

The paper is organized as follows. In Section \ref{Section00}, we introduce the notion of a \textit{generalized matched pair} (cf. Definition \ref{DefGMP1}) and lay the groundwork by establishing several key results (e.g., Propositions \ref{ConditionsPropoOfDoublSymplExten} and \ref{OmegaIsSymplecticGMP}), which will be used in the subsequent sections. In Section \ref{Ar4sect1}, we define the symplectic double extension by a line for an arbitrary symplectic Lie algebra. We then establish that any symplectic Lie algebra which has a $1$-dimensional ideal can be obtained through this construction (cf. Theorem \ref{TheoDoubExtdim1}). Additionally, we provide necessary and sufficient conditions for the symplectic double extension by a line to preserve the unimodularity or the solvability of the symplectic Lie algebra (cf. Corollary \ref{CorSec2-1UnimSolv}). The section concludes with illustrative examples. In Section \ref{Section3}, we introduce the symplectic double extension by a plane for a symplectic Lie algebra (cf. Proposition \ref{PropSymLiAlgDoblExtDim2Prop0}). Further, we establish necessary and sufficient conditions for a symplectic double extension by a plane to preserve the unimodularity property of the symplectic Lie algebra (cf. Proposition \ref{PropSymLiAlgDoblExtDim2Prop01}). In Section \ref{Section3p}, we demonstrate that any solvable symplectic Lie algebra which has an isotropic ideal arises as a symplectic double extension of a solvable symplectic Lie algebra by either a line or a plane (cf. Theorem \ref{TheoDoubExtdim2}). Moreover, since any unimodular symplectic Lie algebra is solvable (see Theorem \ref{Chu'sTh}), we establish that any unimodular symplectic Lie algebra which has an isotropic ideal is a symplectic double extension of a unimodular symplectic Lie algebra by either a line or a plane (cf. Corollary \ref{TheoDoubExtdmi1dim2Unimodlr}). In addition, we show that the nondegeneracy of the derived ideal is a sufficient condition for a symplectic Lie algebra to be unimodular, and consequently solvable (cf. Theorem \ref{MainResNot}). In Section \ref{Art4SectionIrreducible}, we give a new purely algebraic proof of a structural theorem about symplectically irreducible symplectic Lie algebras (cf. Theorem \ref{T1SymplIrredu}). Finally, since symplectically irreducible symplectic Lie algebras do not exist in dimensions $2$ and $4$ (cf. Corollary \ref{CorNotExisDimLes6}), we identify all $6$-dimensional Lie algebras that admit a symplectic structure making them symplectically irreducible (cf. Proposition \ref{SymplecIrreSympLiAlgebOfDim6}).\\

Throughout this paper, for a symplectic Lie algebra $(\g,[\,\cdot\,,\cdot\,],\w)$, we denote by $u^\mathfrak{g}\in\mathfrak{g}$ the unique vector of $\mathfrak{g}$ which satisfies $\omega\left(u^\mathfrak{g},x\right)=\tr\left(\ad_x\right)$, for all $x\in\mathfrak{g}$, and for any linear endomorphism $F\in\gl(\mathfrak{\g})$, we denote by $F^\divideontimes\in\gl(\mathfrak{g})$ the adjoint of $F$ with respect to $\omega$, i.e., $\omega(F^\divideontimes(x),y)=\omega(x,F(y))$. Moreover, for any vector subspace $\mathfrak{m}\subseteq\g$, we denote by $\mathfrak{m}^\perp$ the \textit{symplectic complement} (or the \textit{orthogonal}) of $\mathfrak{m}$, defined as $\mathfrak{m}^\perp:=\left\{x\in\g\,|\,\,\w(x,y)=0,\,\,\forall\,y\in\mathfrak{m}\right\}$. A vector subspace $\mathfrak{m}$ is called \textit{isotropic} (resp. \textit{nondegenerate}) if $\mathfrak{m}\subset\mathfrak{m}^\perp$ (resp. $\mathfrak{m}\cap\mathfrak{m}^\perp=\{0\}$). Furthermore, an ideal $\ii$ of $(\g,\cro)$ is called \textit{normal} if $\ii^\perp$ is also an ideal of $(\g,\cro)$. Finally, we denote by $\kappa_\g:\g\times\g\to\rel$ the Killing form of $(\g,\cro)$.

\section{Symplectic double extension process}\label{Section00}

In this section, we introduce our symplectic double extension process, which will be employed in the subsequent sections. We begin by recalling some necessary background on Lie algebras, including a description of the Chevalley–Eilenberg differential, then we provide our construction.

\begin{definition}
Let $(\g,\cro)$ be a Lie algebra, and $V$ a finite dimensional vector space. A linear map $\LS:\g\to\gl(V)$ is called a \textit{representation} if 
\begin{equation*}
\left[\LS(x),\LS(y)\right]_{\gl(V)}=\LS([x,y]),
\end{equation*}
for any $x,y\in\g$, where $\left[\LS(x),\LS(y)\right]_{\gl(V)}:=\LS(x)\ro\LS(y)-\LS(y)\ro\LS(x)$.
\end{definition}

\begin{example}
For any Lie algebra $(\g,\cro)$, the linear map
\begin{equation*}
\ad:\g\to\gl(\g),\quad\textnormal{ written }\quad x\mapsto \ad_x,
\end{equation*}
with $\ad_x(y):=[x,y]$, is a representation called the \textit{adjoint representation} of $(\g,\cro)$.
\end{example}

Let $\LS:\g\to\gl(V)$ be a representation of $(\g,\cro)$ in a finite dimensional vector space $V$. For each $k\in\nat\backslash\{0\}$, set
\begin{equation*}
\mathcal{C}^k(\g,\LS):=\Bigg\{\theta:\underbrace{\g \times \cdots \times \g}_{k \text { copies }}\rightarrow V\,\,\Big|\,\, \theta\textnormal{ is a multilinear alternating map}\Bigg\},
\end{equation*}
and define a linear map $\delta_{\LS}:\mathcal{C}^k(\g,\LS)\rightarrow \mathcal{C}^{k+1}(\g,\LS)$, called the \textit{Chevalley-Eilenberg differential} corresponding to the representation $\LS:\g\to\gl(V)$ (see \cite[p. $195$]{HilNeb}) by:
\begin{eqnarray*}
\left(\delta_{\LS}\theta\right)(x_0,\ldots,x_{k})&:=&\somzpk(-1)^p\LS(x_p)\Big(\theta\big(x_0,\ldots,\widehat{x_p},\ldots,x_k\big)\Big)\\
&&\,\,+\somzpdqk(-1)^{p+q}\theta\big([x_p,x_q],x_0,\ldots,\widehat{x_p},\ldots,\widehat{x_q},\ldots,x_k\big),\nonumber
\end{eqnarray*}
for $\theta\in\mathcal{C}^k(\g,\LS)$ and $x_0,\ldots,x_k\in\g$, where the hats indicate omitted arguments. We denote by $\mathcal{Z}^k(\g,\LS)$ the vector space of $k$-cocycles of $\g$ given by:
\begin{equation*}
\mathcal{Z}^k(\g,\LS):=\Big\{\theta\in\mathcal{C}^k(\g,\LS)\,\,\big|\,\,\,\delta_{\LS}\theta=0\Big\}.
\end{equation*}
For our purposes we will only focus on the particular cases when $k=1$ or $k=2$. For $\theta_1\in\mathcal{C}^1(\g,\LS)$ and $x,y\in\g$, one has:
\begin{equation}\label{diff1}
\left(\delta_{\LS}\theta_1\right)(x,y)=\LS(x)\big(\theta_1(y)\big)-\LS(y)\big(\theta_1(x)\big)-\theta_1([x,y]).
\end{equation}
Similarly, for $\theta_2\in\mathcal{C}^2(\g,\LS),\,z\in\g$, we have:
\begin{eqnarray}\label{diff2}
\left(\delta_{\LS}\theta_2\right)(x,y,z)&=&\LS(x)\big(\theta_2(y,z)\big)-\LS(y)\big(\theta_2(x,z)\big)+\LS(z)\big(\theta_2(x,y)\big)\nonumber\\
&&-\,\theta_2([x,y],z)+\theta_2([x,z],y)-\theta_2([y,z],x) \nonumber\\
&=&\sum_{\circlearrowright (x,y,z)}\LS(x)\big(\theta_2(y,z)\big)+\sum_{\circlearrowright(x,y,z)}\theta_2(x,[y,z]),
\end{eqnarray}
where $\displaystyle\sum_{\circlearrowright (x,y,z)}$ denotes summation over the set of cyclic permutations of $x,y$ and $z$. If $V=\rel$, and $\LS:\g\to\gl(\rel)$ is the trivial representation, then we simply write $\mathcal{Z}^k(\g,\rel)$ (resp. $\delta$) instead of $\mathcal{Z}^k(\g,\LS)$ (resp. $\delta_{\LS}$), and in this case the formulas \eqref{diff1} and \eqref{diff2} become:
\begin{equation*}
(\delta\theta_1)(x,y)=-\theta_1([x,y]),\et (\delta\theta_2)(x,y,z)=\sum_{\circlearrowright(x,y,z)}\theta_2(x,[y,z]).
\end{equation*}

We end the preliminaries by the following definition:

\begin{definition}
Let $V$ be a vector space, $F\in\gl(V)$ a linear endomorphism of $V$, $\alpha,\beta\in V^*$ two linear forms, and $\varphi\in\bigwedge^2V^*$ a skew-symmetric bilinear form. We define the wedge product $\alpha\wedge\beta$, $\varphi\wedge\alpha$, and the bracket $[\![\alpha, F]\!] \in\bigwedge^2V^*\otimes V$ as follows: 
\begin{eqnarray*}
(\alpha\wedge\beta)(u,v)&:=&\alpha(u)\beta(v)-\alpha(v)\beta(u);\\
(\varphi\wedge\alpha)(u,v,w)&:=&\varphi(u,v)\alpha(w)+\varphi(w,u)\alpha(v)+\varphi(v,w)\alpha(u);\\
\left[\![\alpha, F]\!\right](u,v)&:=&\alpha(u)F(v)-\alpha(v)F(u),
\end{eqnarray*}
for any $u,v,w\in V$.
\end{definition}

In the literature, there exists the notion of a \textit{matched pair} (see \cite{LW,Maj}), which refers to the direct sum of two Lie algebras as vector spaces, equipped with a Lie bracket such that both algebras are Lie subalgebras, though generally neither is an ideal. To describe solvable symplectic Lie algebras, which are the subject of this paper, we need a slightly different and more general notion. This notion is the \textit{generalized matched pair}, which we define below.

\begin{definition}\label{DefGMP1}
A \textit{generalized matched pair} is a quintuple $(\kl,\h,\lk,\lh,\Phi)$ in which $(\kl,\cro_\kl)$, $(\h,\cro_\h)$ are both Lie algebras together with linear maps $\operatorname{L}^\kl:\kl\to\gl(\h)$, $\operatorname{L}^\h:\h\to\gl(\kl)$ and skew-symmetric bilinear map $\Phi:\kl\times\kl\to\h$ such that:
\begin{itemize}
\item[]$(\textnormal{GMP}1)$\quad  $\lh:\h\to\gl(\kl)$ is a representation of $(\h,\cro_\h)$, i.e., for any $h_1,h_2\in\h$
\begin{equation*}
\lh\big([h_1,h_2]_\h\big)=\,\left[\lh(h_1),\lh(h_2)\right]_{\gl(\kl)};
\end{equation*}
\item[]$(\textnormal{GMP}2)$\quad  $\lk:\kl\to\gl(\h),\,\lh:\h\to\gl(\kl)$, and $\Phi:\kl\times\kl\to\h$ satisfy the following five compatibility conditions:
\begin{equation*}
	\left\{
	\begin{array}{r c l}
		\displaystyle\sum_{\circlearrowright (b_1,b_2,b_3)}\lh\Big(\Phi(b_1,b_2)\Big)(b_3)&=&0,\\\\ 
		\displaystyle\sum_{\circlearrowright (b_1,b_2,b_3)}\Phi\big(b_1,[b_2,b_3]_\kl\big)&=&-\displaystyle\sum_{\circlearrowright (b_1,b_2,b_3)}\lk(b_1)\Big(\Phi(b_2,b_3)\Big),\hfill\\\\
	 \lk\big([b_1,b_2]_\kl\big)(h)-\left[\lk(b_1),\lk(b_2)\right]_{\gl(\h)}(h)&=&\Big[h,\Phi_1(b_1,b_2)\Big]_\h-\Phi\Big(\lh(h)(b_1),b_2\Big)-\Phi\Big(b_1,\lh(h)(b_2)\Big),\hfill\\\\
	 \lh\Big(\lk(b_2)(h)\Big)(b_1)-\lh\Big(\lk(b_1)(h)\Big)(b_2)&=&\lh(h)\big([b_1,b_2]_\kl\big)-\Big[\lh(h)(b_1),b_2\Big]_\kl-\Big[b_1,\lh(h)(b_2)\Big]_\kl,\hfill\\\\
	 \lk\Big(\lh(h_2)(b)\Big)(h_1)-\lk\Big(\lh(h_1)(k)\Big)(h_2)&=&\lk(b)\big([h_1,h_2]_\h\big)-\Big[\lk(b)(h_1),h_2\Big]_\h-\Big[h_1,\lk(b)(h_2)\Big]_\h,\hfill
	\end{array}
	\right.
\end{equation*}
for any $b,b_1,b_2,b_3\in\kl$, and $h,h_1,h_2\in\h$.
\end{itemize}
\end{definition}

The following theorem, which constitutes our first result of this section, provides justification for the terminology introduced in Definition \ref{DefGMP1}.

\begin{theorem}\label{ThmGMP1}
Let $(\kl,\cro_\kl)$, $(\h,\cro_\h)$ be two Lie algebras, $\lk:\kl\to\gl(\h)$, $\lh:\h\to\gl(\kl)$ are two linear maps, and $\Phi:\kl\times\kl\to\h$ is a skew-symmetric bilinear map. Define on the vector space $\g:=\kl\oplus\h$ a bracket $\cro\in\bigwedge^2\g^*\otimes\g$ by:
\begin{eqnarray}
\Big[b_1+h_1,b_2+h_2\Big]&:=&[b_1,b_2]_\kl+\lh(h_1)(b_2)-\lh(h_2)(b_1)\nonumber\\
&&+\,[h_1,h_2]_\h+\lk(b_1)(h_2)-\lk(b_2)(h_1)+\Phi(b_1,b_2).
\end{eqnarray}
Then, $(\g,\cro)$ is a Lie algebra if and only if $(\kl,\h,\lk,\lh,\Phi)$ is a generalized matched pair.\vspace{2mm}
\end{theorem}

\begin{proof}
It suffices to show that for any $b,b_1,b_2,b_3\in\kl$, and $h,h_1,h_2\in\h$, the three equalities
\begin{equation}\label{ProofGMPThm}
\displaystyle\sum_{\circlearrowright (b_1,b_2,b_3)}\big[b_1,[b_2,b_3]\big]=0,\quad\displaystyle\sum_{\circlearrowright (b_1,b_2,h)}\big[b_1,[b_2,h]\big]=0,\et\displaystyle\sum_{\circlearrowright (b,h_1,h_2)}\big[b,[h_1,h_2]\big]=0\quad
\end{equation}
are satisfied if and only if $(\kl,\h,\lk,\lh,\Phi)$ is a generalized matched pair, i.e., $(\textnormal{GMP}1)$ and $(\textnormal{GMP}2)$ in Definition \ref{DefGMP1} are fulfilled. However, by a direct computation, it is easy to verify that the first equality in \eqref{ProofGMPThm} is equivalent to the first two conditions of $(\textnormal{GMP}2)$, the second equality is equivalent to the third and fourth conditions of $(\textnormal{GMP}2)$, and the third equality is equivalent to $(\textnormal{GMP}1)$ together with the last condition of $(\textnormal{GMP}2)$.
\end{proof}

\begin{remark}\label{Rmq1GenBicrProd}
The Lie algebra $(\g,\cro)$ is called a \textit{generalized bicrossed product} of $(\h,\cro_\h)$ and $(\kl,\cro_\kl)$. One can observe that $(\h,\cro_\h)$ is a Lie subalgebra of $(\g,\cro)$, whereas $(\kl,\cro_\kl)$ generally is not. In fact, $(\kl,\cro_\kl)$ is a Lie subalgebra of $(\g,\cro)$ if and only if $\Phi = 0$, in which case $(\kl,\h,\lk,\lh)$ constitutes a matched pair (cf. \cite{LW,Maj}).
\end{remark}

Now, let us introduce our symplectic double extension. Let $(\ai,\cro_\ai,\w_\ai)$ be a symplectic Lie algebra, and $(\bi,\cro_\bi)$ a Lie algebra. We start by a generalized semi-direct product of the dual $\bi^*$ (regarded as an abelian Lie algebra) by $(\ai,\cro_\ai)$. More precisely, we have:

\begin{proposition}\label{PropoGeneSemDirProd}
Let $\rho_1:\ai\to\gl(\bi^*)$ be a linear map, and $\psi:\ai\times\ai\to\bi^*$ a skew-symmetric bilinear map. Define on the vector space $\h:=\ai\oplus\bi^*$ a bracket $\cro_\h\in\bigwedge^2\h^*\otimes\h$ as follows:
\begin{equation}\label{LieBrackPropoGeneSemDirProd}
\Big[a_1+f_1,a_2+f_2\Big]_\h=[a_1,a_2]_\ai+\psi(a_1,a_2)+\rho_1(a_1)\big(f_2\big)-\rho_1(a_2)\big(f_1\big),
\end{equation}
for all $a_1,a_2\in\ai$ and $f_1,f_2\in\bi^*$. Then, $\cro_\h$ is a Lie bracket on $\h$ if and only if $\rho_1:\ai\to\gl(\bi^*)$ is a representation and $\psi\in\mathcal{Z}^2(\ai,\rho_1)$.
\end{proposition}

\begin{proof}
For any $a_1,a_2\in\ai$, and $f\in\bi^*$, a straightforward computation yields that
\begin{equation*}
\displaystyle\sum_{\circlearrowright (a_1,a_2,f)}\Big[a_1,\big[a_2,f\big]_\h\Big]_\h=\Big[\rho_1(a_1),\rho_1(a_2)\Big]_{\gl(\bi^*)}(f)-\rho_1\big([a_1,a_2]_\ai\big)(f).
\end{equation*}
Analogously, for any $a_1,a_2,a_3\in\ai$, we have:
\begin{equation*}
\displaystyle\sum_{\circlearrowright (a_1,a_2,a_3)}\Big[a_1,\big[a_2,a_3\big]_\h\Big]_\h=\left(\delta_{\rho_1}\psi\right)(a_1,a_2,a_3).
\end{equation*}
Hence, the equivalence follows.
\end{proof}

\begin{remark}\label{Rmq1GenSemiDirProd}
The Lie algebra $(\h:=\ai\oplus\bi^*,\cro_\h)$ is called a \textit{generalized semi-direct product} of the abelian Lie algebra $\bi^*$ by the Lie algebra $(\ai,\cro_\ai)$, which means that in the Lie algebra $(\h,\cro_\h)$, $\bi^*$ is an ideal, while $\ai$ is generally not a Lie subalgebra. In addition, one can easily see that $(\ai,\bi^*,\rho_1,0,\psi)$ is a generalized matched pair, and that its corresponding Lie algebra introduced in Theorem \ref{ThmGMP1} is $(\h,\cro_\h)$.
\end{remark}

In the next step, we construct a symplectic Lie algebra structure on the vector space $\g := \bi \oplus \ai \oplus \bi^*$ in such a way that $(\ai, \cro_\ai, \w_\ai)$ becomes a symplectic Lie subalgebra of it. Let $\rho_1:\ai\to\gl(\bi^*)$ be a representation, and $\psi\in\mathcal{Z}^2(\ai,\rho_1)$ a $2$-cocycle. Then, by Proposition \ref{PropoGeneSemDirProd}, $(\h:=\ai\oplus\bi^*,\cro_\h)$ is a Lie algebra, where $\cro_\h$ is defined by \eqref{LieBrackPropoGeneSemDirProd}. We endow the vector space $\g=\bi\oplus\h$ with the following bracket:
\begin{equation}\label{LiBraDobExtensGeneral}
	\left\{
	\begin{array}{r c l}\vspace{1mm}
		\left[b_1,b_2\right]&=&[b_1,b_2]_\bi+\Phi_1(b_1,b_2)+\Phi_2(b_1,b_2),\\\vspace{1mm}
		\left[a_1,a_2\right]&=&\big[a_1,a_2\big]_\h,\\\vspace{1mm}
		\left[a,f\right]&=&\big[a,f\big]_\h,\\\vspace{1mm}
		\left[b,f\right]&=&\rho_2(b)(f),\hfill\\
		\left[b,a\right]&=&-\pi_1(a)(b)+\pi_2(b)(a)+\Omega(b,a),
	\end{array}
	\right.\\
\end{equation}
where the unspecified brackets are given by skew-symmetry, with $\pi_1:\ai\to\gl(\bi)$, $\pi_2:\bi\to\gl(\ai)$, $\rho_2:\bi\to\gl(\bi^*)$ are linear maps, $\Phi_1:\bi\times\bi\to\ai$, $\Phi_2:\bi\times\bi\to\bi^*$ are skew-symmetric bilinear maps, and $\Omega:\bi\times\ai\to\bi^*$ is a bilinear map. Our second result of this section is:

\begin{proposition}\label{ConditionsPropoOfDoublSymplExten}
The vector space $\g=\bi\oplus\ai\oplus\bi^*$ endowed with the bracket $\cro\in\bigwedge^2\g^*\otimes\g$ given by \eqref{LiBraDobExtensGeneral} is a Lie algebra if and only if the following fifth conditions are satisfied:
\begin{itemize}
\item[$(\textnormal{C$1$})$] For any $a\in\ai$ and $b\in\bi$, we have:
\begin{equation*}
\Big[\rho_1(a),\rho_2(b)\Big]_{\gl(\bi^*)}=\,\,\rho_2\Big(\pi_1(a)(b)\Big)-\rho_1\Big(\pi_2(b)(a)\Big)\,;
\end{equation*}
\item[$(\textnormal{C$2$})$] For any $b_1,b_2,b_3\in\bi$, we have:\vspace{1mm}
\begin{equation*}
	\left\{
\begin{array}{r c l}
\displaystyle\sum_{\circlearrowright (b_1,b_2,b_3)}\pi_1\Big(\Phi_1(b_1,b_2)\Big)(b_3)\hfill&=&0,\\\\ 
\displaystyle\sum_{\circlearrowright (b_1,b_2,b_3)}\Phi_1\big(b_1,[b_2,b_3]_\bi\big)\hfill&=&-\displaystyle\sum_{\circlearrowright (b_1,b_2,b_3)}\pi_2(b_1)\Big(\Phi_1(b_2,b_3)\Big),\hfill\\\\
\displaystyle\sum_{\circlearrowright (b_1,b_2,b_3)}\Phi_2\big(b_1,[b_2,b_3]_\bi\big)\hfill&=&-\displaystyle\sum_{\circlearrowright (b_1,b_2,b_3)}\rho_2(b_1)\Big(\Phi_2(b_2,b_3)\Big)-\displaystyle\sum_{\circlearrowright (b_1,b_2,b_3)}\Omega\Big(b_1,\Phi_1(b_2,b_3)\Big)\,;\hfill
\end{array}
\right.
\end{equation*}
\item[$(\textnormal{C$3$})$] For any $a\in\ai$ and $b_1,b_2\in\bi$, we have:\vspace{1mm}
\begin{equation*}
	\left\{
\begin{array}{r c l} \pi_1(a)\big([b_1,b_2]_\bi\big)-\Big[\pi_1(a)(b_1),b_2\Big]_\bi-\Big[b_1,\pi_1(a)(b_2)\Big]_\bi\hfill&=&\pi_1\Big(\pi_2(b_2)(a)\Big)(b_1)-\pi_1\Big(\pi_2(b_1)(a)\Big)(b_2),\\\\ 
\Phi_1\Big(\pi_1(a)(b_1),b_2\Big)+\Phi_1\Big(b_1,\pi_1(a)(b_2)\Big)-\Big[a,\Phi_1(b_1,b_2)\Big]_\ai&=&\Big[\pi_2(b_1),\pi_2(b_2)\Big]_{\gl(\ai)}(a)-\pi_2\big([b_1,b_2]_\bi\big)(a),\hfill\\\\\vspace{1mm}
\psi\Big(a,\Phi_1(b_1,b_2)\Big)+\rho_2(b_1)\Big(\Omega(b_2,a)\Big)-\rho_2(b_2)\Big(\Omega(b_1,a)\Big)&=&\Phi_2\Big(\pi_1(a)(b_1),b_2\Big)+\Phi_2\Big(b_1,\pi_1(a)(b_2)\Big)\hfill\\\vspace{1mm}
&&-\,\rho_1(a)\Big(\Phi_2(b_1,b_2)\Big)+\Omega\big([b_1,b_2]_\bi,a\big)\\
&&-\,\Omega\Big(b_1,\pi_2(b_2)(a)\Big)+\Omega\Big(b_2,\pi_2(b_1)(a)\Big)\,;
\end{array}
\right.
\end{equation*}
\item[$(\textnormal{C$4$})$] For any $b_1,b_2\in\bi$, we have:\vspace{1mm}
\begin{equation*}
\Big[\rho_2(b_1),\rho_2(b_2)\Big]_{\gl(\bi^*)}-\rho_2\big([b_1,b_2]_\bi\big)=\,\rho_1\Big(\Phi_1(b_1,b_2)\Big)\,;\vspace{1mm}
\end{equation*}
\item[$(\textnormal{C$5$})$] The linear map $\pi_1:\ai\to\gl(\bi)$ is a representation, and for any $a_1,a_2\in\ai$, $b\in\bi$, we have:\vspace{1mm}
\begin{equation*}
\left\{
\begin{array}{r c l} \pi_2(b)\big([a_1,a_2]_\ai\big)-\Big[\pi_2(b)(a_1),a_2\Big]_\ai-\Big[a_1,\pi_2(b)(a_2)\Big]_\ai\hfill&=&\pi_2\Big(\pi_1(a_2)(b)\Big)(a_1)-\pi_2\Big(\pi_1(a_1)(b)\Big)(a_2),\\\\\vspace{1mm}
\psi\Big(\pi_2(b)(a_1),a_2\Big)+\psi\Big(a_1,\pi_2(b)(a_2)\Big)-\rho_2(b)\Big(\psi(a_1,a_2)\Big)&=&\rho_1(a_2)\Big(\Omega(b,a_1)\Big)-\rho_1(a_1)\Big(\Omega(b,a_2)\Big)\\\vspace{1mm}
&&+\,\Omega\big(b,[a_1,a_2]_\ai\big)+\Omega\Big(\pi_1(a_1)(b),a_2\Big)\\
&&-\,\Omega\Big(\pi_1(a_2)(b),a_1\Big).
	\end{array}
	\right.
\end{equation*}
\end{itemize}
\end{proposition}

\begin{proof}
Since the proof consists of straightforward computations and all conditions can be obtained in the same way, we will only establish the first condition (C$1$). For $a\in\ai,\,b\in\bi$, and $f\in\bi^*$, we compute:
\begin{eqnarray*}
\big[b,[a,f]\big]+\big[f,[b,a]\big]+\big[a,[f,b]\big]&=&\Big[b,\rho_1(a)(f)\Big]+\Big[f,-\pi_1(a)(b)+\pi_2(b)(a)\Big]-\Big[a,\rho_2(b)(f)\Big]\\
&=&\rho_2(b)\ro\rho_1(a)(f)+\rho_2\Big(\pi_1(a)(b)\Big)(f)-\rho_1\Big(\pi_2(b)(a)\Big)(f)-\rho_1(a)\ro\rho_2(b)(f)\\
&=&\Big[\rho_2(b),\rho_1(a)\Big]_{\gl(\bi^*)}(f)+\rho_2\Big(\pi_1(a)(b)\Big)(f)-\rho_1\Big(\pi_2(b)(a)\Big)(f).
\end{eqnarray*}
Thus, the first condition (C$1$) is equivalent to $\displaystyle\sum_{\circlearrowright (a,b,f)}\big[b,[a,f]\big]=0$. Similarly, one can check that (C$2$), (C$3$), (C$4$), and (C$5$) are respectively equivalent to $\displaystyle\sum_{\circlearrowright (b_1,b_2,b_3)}\big[b_1,[b_2,b_3]\big]=0$, $\displaystyle\sum_{\circlearrowright (b_1,b_2,a)}\big[b_1,[b_2,a]\big]=0$, $\displaystyle\sum_{\circlearrowright (b_1,b_2,f)}\big[b_1,[b_2,f]\big]=0$,\quad and $\displaystyle\sum_{\circlearrowright (b,a_1,a_2)}\big[b,[a_1,a_2]\big]=0$.
\end{proof}

\begin{remark}
The Lie algebra $(\g=\bi\oplus\ai\oplus\bi^*, \cro)$ is referred to as the \textit{double extension} of $(\ai, \cro_\ai)$ by $(\bi, \cro_\bi)$, as it is constructed via two successive Lie algebra extensions: first, the extension $(\h=\ai\oplus\bi^*, \cro_\h)$, followed by the extension $(\bi \oplus \h, \cro)$. Moreover, if we consider the following two linear maps $\LS^\bi:\bi\to\gl(\h)$, and $\lh:\h\to\gl(\bi)$ defined by:
\begin{equation*}
\LS^\bi(b)\big(a+f\big):=\pi_2(b)(a)+\Omega(b,a)+\rho_2(b)(f),\et\lh\big(a+f\big)(b):=\pi_1(a)(b),
\end{equation*}
and the following skew-symmetric bilinear map
\begin{equation*}
\Phi:\bi\times\bi\to\h,\qquad\Phi(b_1,b_2):=\Phi_1(b_1,b_2)+\Phi_2(b_1,b_2).
\end{equation*}
Then, $(\bi,\h,\LS^\bi,\lh,\Phi)$ is a generalized matched pair, and its associated Lie algebra introduced in Theorem \ref{ThmGMP1} is $(\g,\cro)$.
\end{remark}

It is clear, from \eqref{LiBraDobExtensGeneral}, that $(\bi, \cro_\bi)$ is not a Lie subalgebra of $(\g, \cro)$. However, if we require $(\bi, \cro_\bi)$ to remain a Lie subalgebra, we obtain:

\begin{corollary}\label{CorolGMPSubAlg}
Let $\rho_1:\ai\to\gl(\bi^*),\,\rho_2:\bi\to\gl(\bi^*),\,\pi_1:\ai\to\gl(\bi),\,\pi_2:\bi\to\gl(\ai)$, be representations, and $\psi\in\mathcal{Z}^2(\ai,\rho_1)$. The vector space $\g=\bi\oplus\ai\oplus\bi^*$ endowed with the following bracket:
\begin{equation}\label{LiBraDobExtensParticularCor}
	\left\{
	\begin{array}{r c l} \vspace{1mm}
		\left[a_1,a_2\right]&=&[a_1,a_2]_\ai+\psi(a_1,a_2),\\\vspace{1mm}
		\left[b_1,b_2\right]&=&[b_1,b_2]_\bi,\\\vspace{1mm}
		\left[a,f\right]&=&\rho_1(a)(f),\\\vspace{1mm}
		\left[b,f\right]&=&\rho_2(b)(f),\\
		\left[b,a\right]&=&-\pi_1(a)(b)+\pi_2(b)(a)+\Omega(b,a),
	\end{array}
	\right.\\
\end{equation}
where the other brackets are given by skew-symmetry, is a Lie algebra if and only if for any $a,a_1,a_2\in\ai$, and $b,b_1,b_2\in\bi$, we have:
\begin{equation*}
	\left\{
	\begin{array}{r c l}
		\Big[\rho_1(a),\rho_2(b)\Big]_{\gl(\bi^*)}&=&\rho_2\Big(\pi_1(a)(b)\Big)-\rho_1\Big(\pi_2(b)(a)\Big),\\\\
		 \pi_1(a)\big([b_1,b_2]_\bi\big)-\Big[\pi_1(a)(b_1),b_2\Big]_\bi-\Big[b_1,\pi_1(a)(b_2)\Big]_\bi\hfill&=&\pi_1\Big(\pi_2(b_2)(a)\Big)(b_1)-\pi_1\Big(\pi_2(b_1)(a)\Big)(b_2),\\\\
		\pi_2(b)\big([a_1,a_2]_\ai\big)-\Big[\pi_2(b)(a_1),a_2\Big]_\ai-\Big[a_1,\pi_2(b)(a_2)\Big]_\ai\hfill&=&\pi_2\Big(\pi_1(a_2)(b)\Big)(a_1)-\pi_2\Big(\pi_1(a_1)(b)\Big)(a_2),\\\\
		\Omega\big([b_1,b_2]_\bi,a\big)+\Omega\Big(b_2,\pi_2(b_1)(a)\Big)-\Omega\Big(b_1,\pi_2(b_2)(a)\Big)&=&\rho_2(b_1)\Big(\Omega(b_2,a)\Big)-\rho_2(b_2)\Big(\Omega(b_1,a)\Big),\hfill\\\\\vspace{1mm}
		\psi\Big(\pi_2(b)(a_1),a_2\Big)+\psi\Big(a_1,\pi_2(b)(a_2)\Big)-\rho_2(b)\Big(\psi(a_1,a_2)\Big)&=&\rho_1(a_2)\Big(\Omega(b,a_1)\Big)-\rho_1(a_1)\Big(\Omega(b,a_2)\Big)\\\vspace{1mm}
		&&+\,\Omega\big(b,[a_1,a_2]_\ai\big)+\Omega\Big(\pi_1(a_1)(b),a_2\Big)\\
		&&-\,\Omega\Big(\pi_1(a_2)(b),a_1\Big).
	\end{array}
	\right.
\end{equation*}
\end{corollary}

Now, under the conditions of Proposition \ref{ConditionsPropoOfDoublSymplExten}, we endow the Lie algebra $(\g,\cro)$ with the following nondegenerate skew-symmetric bilinear form:\vspace{1mm}
\begin{equation}\label{OmegDoublExtenSec1}
\w\,\Big(b_1+a_1+f_1,b_2+a_2+f_2\Big):=\,\w_\ai(a_1,a_2)-f_1(b_2)+f_2(b_1),\vspace{1mm}
\end{equation}
for $a_1,a_2\in\ai$, $b_1,b_2\in\bi$, and $f_1,f_2\in\bi^*$. The next proposition provides sufficient and necessary conditions for $\w \in \bigwedge^2 \g^*$ to be a $2$-cocycle of $(\g, \cro)$.

\begin{proposition}\label{OmegaIsSymplecticGMP}
With the previous notations and under the conditions of Proposition \ref{ConditionsPropoOfDoublSymplExten}, the skew-symmetric bilinear form $\w$ given by \eqref{OmegDoublExtenSec1} is a $2$-cocycle of the Lie algebra $(\g,\cro)$ if and only if
the following fifth conditions are satisfied:
\begin{itemize}
	\item[$(\textnormal{C$1$'})$] For any $a\in\ai,\,b\in\bi$, and $f\in\bi^*$, we have:\vspace{1mm}
	\begin{equation*}
		\rho_1(a)(f)(b)=-f\Big(\pi_1(a)(b)\Big)\,;\vspace{1mm}
	\end{equation*}
	\item[$(\textnormal{C$2$'})$] For any $b_1,b_2,b_3\in\bi$, we have:\vspace{1mm}
	\begin{equation*}
	\sum_{\circlearrowright (b_1,b_2,b_3)} \Phi_2(b_1,b_2)(b_3)=0\,;\vspace{1mm}
	\end{equation*}
	\item[$(\textnormal{C$3$'})$] For any $a\in\ai$ and $b_1,b_2\in\bi$, we have:\vspace{1mm}
	\begin{equation*}
		\Omega\big(b_1,a\big)(b_2)-\Omega\big(b_2,a\big)(b_1)=\w_\ai\Big(a,\Phi_1(b_1,b_2)\Big)\,;\vspace{1mm}
	\end{equation*}
	\item[$(\textnormal{C$4$'})$] For any $b_1,b_2\in\bi$, and $f\in\bi^*$, we have:\vspace{1mm}
	\begin{equation*}
		f\big([b_1,b_2]_\bi\big)=\rho_2(b_2)(f)(b_1)-\rho_2(b_1)(f)(b_2)\,;\vspace{1mm}
	\end{equation*}
	\item[$(\textnormal{C$5$'})$] For any $a_1,a_2\in\ai$, and $b\in\bi$, we have:\vspace{1mm}
	\begin{equation*}
	\psi\big(a_1,a_2\big)(b)=\w_\ai\Big(\pi_2(b)(a_1),a_2\Big)+\w_\ai\Big(a_1,\pi_2(b)(a_2)\Big).\vspace{1mm}
	\end{equation*}
\end{itemize}
\end{proposition}

\begin{proof}
Let $a_1,a_2\in\ai$, and $b\in\bi$, we have:
\begin{equation*}
\w\big([a_1,a_2],b\big)+\w\big([b,a_1],a_2\big)-\w\big([b,a_2],a_1\big)=-\psi(a_1,a_2)(b)+\w_\ai\Big(\pi_2(b)(a_1),a_2\Big)-\w_\ai\Big(\pi_2(b)(a_2),a_1\Big).
\end{equation*}
Hence, the condition (C$5$') is equivalent to $(\delta\w)(a_1,a_2,b)=0$. Analogously, we can verify that (C$1$'), (C$2$'), (C$3$'), and (C$4$') are respectively equivalent to $(\delta\w)(a,b,f)=0$, $(\delta\w)(b_1,b_2,b_3)=0$, $(\delta\w)(a,b_1,b_2)=0$, and $(\delta\w)(b_1,b_2,f)=0$.
\end{proof}

\section{Symplectic double extension by a line}\label{Ar4sect1}
In this section, we provide a detailed description of symplectic Lie algebras admitting a $1$-dimensional ideal. Our approach is based on the framework developed in the previous section. In the first part of this section, we apply Corollary \ref{CorolGMPSubAlg} together with Proposition \ref{OmegaIsSymplecticGMP} to construct a symplectic Lie algebra with a $1$-dimensional ideal from an arbitrary symplectic Lie algebra. In the second part, we present a kind of converse to this construction.

Let $(\ai,\cro_\ai, \w_\ai)$ be a symplectic Lie algebra, $a^\ai,b^\ai\in\ai$, $\D\in\gl(\ai)$, and $\varepsilon \in \rel$. We endow the vector space $\g := \rel d_0 \oplus \ai \oplus \rel e_0$, where $\{d_0,e_0\}$ is a basis of $\rel^2$, with the following bracket:
\begin{equation}\label{LiBraDobExIdeDim1Prop0}
	\left\{
	\begin{array}{r c l}\vspace{1mm}
		\left[e_0,d_0\right]&=&\varepsilon\,e_0,\\\vspace{1mm}
		\left[d_0,a\right]&=&\w_\ai(a^\ai,a)d_0+\D(a)+\w_\ai(b^\ai,a)e_0,\\ \vspace{1mm}
		\left[a,e_0\right]&=&\w_\ai(a^\ai,a)e_0,\hfill\\
		\left[a,b\right]&=&[a,b]_\ai+\w_\ai\big(\D(a),b\big)e_0+\w_\ai\big(a,\D(b)\big)e_0,\hfill
	\end{array}
	\right.\\
\end{equation}
where the other brackets are given by skew-symmetry. In addition, we  extend $\w_\ai$ to a nondegenerate skew-symmetric bilinear form $\w\in\bigwedge^2\g^*$ on $\g$ by: 
$$\w:=\w_\ai+d_0^*\wedge e_0^*.$$
Our first result of this section is:

\begin{proposition}\label{PropSymLiAlgDoblExtDim1Prop0}
	With the notations above, $(\g,\cro,\w)$ is a symplectic Lie algebra if and only if
	\begin{itemize}
		\item[$1.$] $	a^\ai\in[\ai,\ai]^\perp\cap\ker\left(\D^\divideontimes-\varepsilon\id_\ai\right)\,;$
		\item[$2.$] $\D$ is a $1$-cocycle of $\ai$ with respect to the following representation:
		\begin{equation*}
		\widetilde{\rho}_{_{a^\ai}}:\ai\longrightarrow\gl(\ai),\qquad a\longmapsto\ad^\ai_a+\w_\ai(a^\ai,a)\id_\ai\,;
		\end{equation*}
		\item[$3.$] The skew-symmetric bilinear form $\varphi:\ai\times\ai\to\rel$ defined by:
		\begin{equation*}
		\varphi(a,b):=\w_\ai\big(\D(a),b\big)+\w_\ai\big(a,\D(b)\big),
		\end{equation*}
for any $a,b\in\ai$, is a $2$-cocycle of $\ai$ with respect to the following representation: 
		\begin{equation*}
		\rho_{_{a^\ai}}:\ai\longrightarrow\gl(\rel),\qquad a\longmapsto \w_\ai(a^\ai,a)\id_\rel\,;
		\end{equation*}
		\item[$4.$] $\varphi_{\D}+\varepsilon\varphi=2\,(i_{b^\ai}\w_\ai)\wedge (i_{a^\ai}\w_\ai)-\delta(i_{b^\ai}\w_\ai)$, where $\varphi_{\D}:\ai\times\ai\to\rel$ is the skew-symmetric bilinear form given by:
		\begin{equation*}
		\varphi_{\D}(a,b):=\varphi\big(\D(a),b\big)+\varphi\big(a,\D(b)\big),\qquad\forall\,a,b\in\ai.
		\end{equation*}
	\end{itemize}
\end{proposition}

\begin{proof}
Regard $\bi:=\rel d_0$ as an abelian Lie algebra, and denote by $\bi^*:=\rel e_0$ its dual. Consider the linear maps $\rho_1:\ai\to\gl(\bi^*),\,\rho_2:\bi\to\gl(\bi^*),\,\pi_1:\ai\to\gl(\bi),\,\pi_2:\bi\to\gl(\ai)$ defined by:
\begin{equation*}
\rho_1(a)(te_0):=\w_\ai(a^\ai,ta)e_0,\quad\rho_2(sd_0)(te_0):=-\varepsilon ste_0,\quad\pi_1(a)(sd_0):=\w_\ai(sa,a^\ai)d_0,\quad\pi_2(sd_0)(a):=s\D(a),
\end{equation*}
for $a\in\ai$, $s,t\in\rel$, and the following two bilinear maps $\psi:\ai\times\ai\to\bi^*$, $\Omega:\bi\times\ai\to\bi^*$
\begin{equation*}
\psi(a,b):=\varphi(a,b)e_0,\qquad\Omega(sd_0,a):=\w_\ai(b^\ai,sa)e_0.
\end{equation*}
A direct computation shows that $\rho_2:\bi \to \gl(\bi^*)$ and $\pi_2:\bi \to \gl(\ai)$ are representations, and that the first, second, and fourth conditions of Corollary \ref{CorolGMPSubAlg}, as well as all five conditions of Proposition \ref{OmegaIsSymplecticGMP}, are satisfied. Moreover, one can easily verify that $\rho_1:\ai \to \gl(\bi^*)$ and $\pi_1:\ai \to \gl(\bi)$ are representations if and only if $a^\ai \in [\ai,\ai]^\perp$. The first condition in Corollary \ref{CorolGMPSubAlg} is equivalent to $a^\ai \in \ker(\D^\divideontimes - \varepsilon \id_\ai)$; the third condition is equivalent to $\D \in \mathcal{Z}^1(\ai, \widetilde{\rho}_{_{a^\ai}})$; the fifth condition is equivalent to the final condition in the proposition. Finally, one can easily observe that $\psi \in \mathcal{Z}^2(\ai, \rho_1)$ if and only if $\varphi \in \mathcal{Z}^2(\ai, \rho_{_{a^\ai}})$.
\end{proof}

\begin{definition}
	A quadruple $(a^\ai,b^\ai,\D,\varepsilon)\in\ai\times\ai\times \gl(\ai)\times\rel$ satisfying the four conditions of Proposition \ref{PropSymLiAlgDoblExtDim1Prop0} is called \textit{admissible} or an \textit{admissible element of} $(\ai,\cro_\ai,\w_\ai)$.
\end{definition}

\begin{definition}
	The symplectic Lie algebra $(\g,\cro,\w)$ given in Proposition \ref{PropSymLiAlgDoblExtDim1Prop0} is called the \textit{symplectic double extension of} $(\ai,\cro_\ai,\w_\ai)$ \textit{by a line}.
\end{definition}

\begin{remark}
	\begin{enumerate}
\item If $a^\ai=0$, then the linear endomorphism $\D\in\gl(\ai)$ is a derivation of the Lie algebra $(\ai,\cro_\ai)$, and $\delta\varphi_{\D}=\delta\varphi=0$.
\item The ideal $\bi^*:=\rel e_0$ of $(\g,\cro)$ is normal, with respect to the symplectic structure $\w$, if and only if $a^\ai=0.$ In this case, the first extension is a central extension of the Lie algebra $(\ai,\cro_\ai)$ by means of $\varphi\in\mathcal{Z}^2(\ai,\rel)$, and the second extension is the semi-direct product of this central extension by the $1$-dimensional abelian Lie algebra $\bi:=\rel d_0$ by means of the derivation $\D\in\mathfrak{der}(\ai)$. Hence, we recover the symplectic (normal) double extension introduced in \cite{MR-Q}.
	\end{enumerate}
\end{remark}

The remainder of this section, which is the second part of this section, is dedicated to proving that any symplectic Lie algebra which has a $1$-dimensional ideal can be realized as a symplectic double extension by a line of a symplectic Lie algebra.

Let $(\g,[\,\cdot\,,\cdot\,],\w)$ be a symplectic Lie algebra, and assume that there exists an element $e_0\in\g\backslash\{0\}$ such that $\mathfrak{i}:=\rel e_0$ is an ideal of $\g$. For $d_0\in\g\backslash\{0\}$ such that $\w(d_0,e_0)=1$, we can decompose $\g$ as a vector space direct sum, i.e.,
\begin{equation*}
	\g=\rel d_0\oplus\ai\oplus\rel e_0,
\end{equation*}
where $\ai\cong\mathfrak{i}^\perp/\mathfrak{i}$, $\w_\ai:=\w_{|\ai\times\ai}\in\bigwedge^2\ai^*$ is nondegenerate, and $\ai^\perp=\Span_\rel\{d_0,e_0\}$. For $a,b\in\ai$, we write:
\begin{equation}\label{eqSolDim1}
	\left\{
	\begin{array}{r c l}\vspace{1mm}
		\left[d_0,e_0\right]&=&\varepsilon_0e_0,\\\vspace{1mm}
		\left[d_0,a\right]&=&f(a)d_0+\D(a)+\ell(a)e_0,\\ \vspace{1mm}
		\left[e_0,a\right]&=&\lambda(a)e_0,\hfill\\
		\left[a,b\right]&=&[a,b]_\ai+\varphi(a,b)e_0,\hfill
	\end{array}
	\right.\\
\end{equation}
where $f,\ell,\lambda\in\ai^*,\,\D\in\gl(\ai),\,[\,\cdot\,,\cdot\,]_\ai\in\bigwedge^2\ai^*\otimes\ai,\,\varphi\in\bigwedge^2\ai^*,$ and $\varepsilon_0\in\rel$. The next proposition provides all the consequences of $\w$ being a $2$-cocycle of $(\g,\cro)$.

\begin{proposition}
	The following equalities hold:
	\begin{itemize}
		\item[$1.$] $f(a)=\w([d_0,a],e_0)=-\lambda(a)\,;$\vspace{1mm}
		\item[$2.$] $\ell(a)=\w(d_0,[d_0,a])\,;$\vspace{1mm}
		\item[$3.$] $\w\big(\D^k(a),b\big)=\w\big(\ad_{d_0}^k(a),b\big)$,\, for any $k\in\nat\backslash\{0\}\,;$\vspace{1mm}
		\item[$4.$] $\varphi(a,b)=\w(d_0,[a,b])=\w_\ai\big(\D(a),b\big)+\w_\ai\big(a,\D(b)\big)$.\vspace{1mm}
	\end{itemize}
\end{proposition}

\begin{proof}
	A direct computation, using the fact that $\w$ is a $2$-cocycle of $(\g,\cro)$.
\end{proof}
Thus, the equalities in \eqref{eqSolDim1} become:
\begin{equation}\label{eqSolDim1II}
	\left\{
	\begin{array}{r c l}\vspace{1mm}
		\left[d_0,e_0\right]&=&\varepsilon_0e_0,\\\vspace{1mm}
		\left[d_0,a\right]&=&f(a)d_0+\D(a)+\ell(a)e_0,\\ \vspace{1mm}
		\left[e_0,a\right]&=&-f(a)e_0,\hfill\\
		\left[a,b\right]&=&[a,b]_\ai+\varphi(a,b)e_0,\hfill
	\end{array}
	\right.\\
\end{equation}
where $\varphi(a,b):=\w_\ai\big(\D(a),b\big)+\w_\ai\big(a,\D(b)\big)$, for any $a,b\in\ai$.\vspace{2mm}

The following proposition summarizes all the consequences of the fact that $[\,\cdot\,,\cdot\,]$ is a Lie bracket on $\g$.

\begin{proposition}\label{propCondDoublExtDim1}
	The following assersions hold:\vspace{1mm}
	\item[$1.$] $(\ai,[\,\cdot\,,\cdot\,]_\ai,\w_\ai)$ is a symplectic Lie algebra$\,;$\vspace{1mm}
	\item[$2.$] $\delta f=0$, i.e., $f\in\mathcal{Z}^1(\ai,\rel)\,;$\vspace{1mm}
	\item[$3.$] $\delta\varphi=-\varphi\wedge f\,;$\vspace{1mm}
	\item[$4.$] For any $a,b\in\ai$, we have:
	\begin{equation}\label{LiealG}
		\left\{
		\begin{array}{r c l}
			\D([a,b]_\ai)-\left[\D(a),b\right]_\ai-\left[a,\D(b)\right]_\ai=f(a)\D(b)-f(b)\D(a),\hfill\hfill\\\\
			\varphi\big(\D(a),b\big)+\varphi\big(a,\D(b)\big)+2\Big(f(a)\ell(b)-f(b)\ell(a)\Big)-\varepsilon_0\varphi(a,b)=\ell([a,b]_\ai),\hfill\hfill\\\\
			f\ro \D=-\varepsilon_0f.\hfill\hfill\hfill\hfill
		\end{array}
		\right.\\
	\end{equation}
\end{proposition}

\begin{proof}
	For $a,b,c\in\ai$, we compute:
	\begin{eqnarray*}
		\sum_{\circlearrowright (a,b,c)} \big[a,[b,c]\big]&=&\sum_{\circlearrowright (a,b,c)} \Big[a,[b,c]_\ai+\varphi(b,c)e_0\Big]\\
		&=&\sum_{\circlearrowright (a,b,c)} \big[a,[b,c]_\ai\big]+\sum_{\circlearrowright (a,b,c)} \varphi(b,c)[a,e_0]\\
		&=&\sum_{\circlearrowright (a,b,c)} \big[a,[b,c]_\ai\big]_\ai+\sum_{\circlearrowright (a,b,c)} \varphi(a,[b,c]_\ai)e_0+\sum_{\circlearrowright (a,b,c)}\varphi(b,c)f(a)e_0\\
		&=&\sum_{\circlearrowright (a,b,c)} \big[a,[b,c]_\ai\big]_\ai+\big(\delta\varphi+\varphi\wedge f\big)(a,b,c)e_0.
	\end{eqnarray*}
	Further, through similar computations, we derive that:
	\begin{equation*}
		\sum_{\circlearrowright (e_0,a,b)} \big[e_0,[a,b]\big]=-f([a,b]_\ai)e_0.
	\end{equation*}
	Moreover, we have:
	\begin{equation*}
		\begin{array}{r c l}\vspace{0.75mm}
		\big[d_0,[a,b]\big]-\big[[d_0,a],b\big]-\big[a,[d_0,b]\big]&=&\Big[d_0,[a,b]_\ai+\varphi(a,b)e_0\Big]-\Big[f(a)d_0+\D(a)+\ell(a)e_0,b\Big]\\\vspace{1.5mm}
		&&-\,\Big[a,f(b)d_0+\D(b)+\ell(b)e_0\Big]\\\vspace{0.75mm}
		&=&\big[d_0,[a,b]_\ai\big]+\varepsilon_0\varphi(a,b)e_0-f(a)[d_0,b]-[\D(a),b]\\\vspace{1.5mm}
		&&-\,\ell(a)[e_0,b]+f(b)[d_0,a]-[a,\D(b)]+\ell(b)[e_0,a]\\\vspace{0.75mm}
		&=&f([a,b]_\ai)d_0+\D([a,b]_\ai)+\Big(\ell([a,b]_\ai)+\varepsilon_0\varphi(a,b)\Big)e_0\\\vspace{0.75mm}
		&&-\,f(a)\D(b)-f(a)\ell(b)e_0-[\D(a),b]_\ai-\varphi(\D(a),b)e_0\\\vspace{0.75mm}
		&&+\,\ell(a)f(b)e_0+f(b)\D(a)+f(b)\ell(a)e_0-[a,\D(b)]_\ai\\\vspace{1.5mm}
		&&-\,\varphi(a,\D(b))e_0-\ell(b)f(a)e_0\\\vspace{0.75mm}
		&=&f([a,b]_\ai)d_0+\D([a,b]_\ai)-[\D(a),b]_\ai-[a,\D(b)]_\ai\\\vspace{0.75mm}
		&&-\,f(a)\D(b)+f(b)\D(a)+\Big(\ell([a,b]_\ai)+\varepsilon_0\varphi(a,b)\\
		&&-\,\varphi(\D(a),b)-\varphi(a,\D(b))-2f(a)\ell(b)+2f(b)\ell(a)\Big)e_0.
		\end{array}
	\end{equation*}
	Finally, we have:
	\begin{eqnarray*}
		\big[[d_0,e_0],a\big]+\big[e_0,[d_0,a]\big]-\big[d_0,[e_0,a]\big]&=&-2\varepsilon_0f(a)e_0-f(\D(a))e_0+\varepsilon_0f(a)e_0\\
		&=&-\,\Big(f(\D(a))+\varepsilon_0f(a)\Big)e_0,
	\end{eqnarray*}
and this completes the proof.\vspace{1mm}
\end{proof}

\begin{definition}
	The symplectic Lie algebra $(\ai,[\,\cdot\,,\cdot\,]_\ai,\w_\ai)$ given in Proposition \ref{propCondDoublExtDim1} is called the \textit{symplectic reduction} of $(\g,[\,\cdot\,,\cdot\,],\w)$ with respect to the isotropic ideal $\mathfrak{i}$.
\end{definition}

Let us see that the conditions of Proposition \ref{propCondDoublExtDim1} gives rise to an admissible element $(a^\ai,b^\ai,\D,\varepsilon)$. First, since $f\in\mathcal{Z}^1(\ai,\rel)$, we can define a representation $\rho_{_f}:\ai\to\gl(\rel)$ of $\ai$ by setting $\rho_{_f}(a):=f(a)\id_\rel$. If we denote by $\delta_{_f}$ the Chevalley-Eilenberg differential corresponding to it, then it is clear that the third condition of Proposition \ref{propCondDoublExtDim1} is equivalent to $\delta_{_f}\varphi=0$. Moreover, we can define another representation $\widetilde{\rho}_{_f}:\ai\to\gl(\ai)$ of $\ai$ by $\widetilde{\rho}_{_f}(a):=\ad^\ai_a+f(a)\id_\ai$. So, if we denote by $\widetilde{\delta}_{_f}$ its Chevalley-Eilenberg differential, then one can easily check that the first equation in \eqref{LiealG} is exactly $\widetilde{\delta}_{_f}\D=0$. On the other hand, if we define $\varphi_{\D}\in\bigwedge^2\ai^*$ by:
\begin{equation*}
\varphi_{\D}(a,b):=\varphi\big(\D(a),b\big)+\varphi\big(a,\D(b)\big),\qquad\forall\,a,b\in\ai.
\end{equation*}
Then, the second equation in \eqref{LiealG} becomes 
\begin{eqnarray*}
	\varphi_{\D}(a,b)-\varepsilon_0\varphi(a,b)&=&-2\Big(f(a)\ell(b)-f(b)\ell(a)-\ell([a,b]_\ai)\Big)-\ell([a,b]_\ai)\\
	&=&-\,\Big(2\,\delta_{_f}\ell-\delta\ell\Big)(a,b).
\end{eqnarray*}
In other words, $\varphi_{\D}-\varepsilon_0\varphi=\left(\delta-2\,\delta_{_f}\right)\ell$, which can also be written as
\begin{eqnarray*}
	\varphi_{\D}-\varepsilon_0\varphi\,=\,2\,\ell\wedge f-\delta\ell\,=\,\ell\wedge f-\delta_{_f}\ell.
\end{eqnarray*}
In addition, since $\varphi\in\mathcal{Z}^2(\ai,\rho_{_f})$, we also have: $$\delta_{_f}\varphi_{\D}=\delta_{_f}\ell\wedge f=\delta\ell\wedge f.$$
For the third equation in \eqref{LiealG}, since $\w_\ai$ is nondegenerate, there exists an element $a_f\in[\ai,\ai]^\perp$ (because $\delta f=0$) such that $f=i_{a_f}\w_\ai$. Thus,
\begin{eqnarray*}
	f\ro\D=-\varepsilon_0 f&\Leftrightarrow&\w_\ai\big(a_f,\D(a)\big)=-\varepsilon_0\,\w_\ai\big(a_f,a\big),\quad\forall\,a\in\ai\\
	&\Leftrightarrow& \w_\ai\left(\D^\divideontimes(a_f),a\right)=\,\w_\ai\left(-\varepsilon_0a_f,a\right),\quad\forall\,a\in\ai\\
	&\Leftrightarrow&\D^\divideontimes(a_f)=-\varepsilon_0a_f.
\end{eqnarray*}
Finally, let $b_\ell$ be the unique element of $\ai$ such that $\ell=i_{b_\ell}\w_\ai$, and $(a^\ai,b^\ai,\D,\varepsilon):=(a_f,b_\ell,\D,-\varepsilon_0)$. We have established the following:

\begin{proposition}
Using the notation introduced above, $(a^\ai,b^\ai,\D,\varepsilon)$ is an admissible element of the symplectic Lie algebra $(\ai,\cro_\ai,\w_\ai)$ given in Proposition $\ref{propCondDoublExtDim1}$.
\end{proposition}

In summary, our main theorem of this section is: 

\begin{theorem}\label{TheoDoubExtdim1}
	Any symplectic Lie algebra $(\g,\cro,\w)$ which admits a $1$-dimensional ideal is a symplectic double extension by a line of a symplectic Lie algebra $(\ai,\cro_\ai,\w_\ai)$. Conversely, let $(\ai,[\,\cdot\,,\cdot\,]_\ai,\w_\ai)$ be a symplectic Lie algebra, and $(a^\ai,b^\ai,\D,\varepsilon)$ an admissible element of it. Then, the symplectic double extension by a line of $(\ai,\cro_\ai,\w_\ai)$ is a symplectic Lie algebra which has a $1$-dimensional ideal.
\end{theorem}

In the case where the $1$-dimensional ideal is central, we have:

\begin{corollary}
	If the $1$-dimensional ideal $\rel e_0$ is central, i.e., $e_0\in\mathfrak{z}(\g)$, then $a^\ai=0$, $\varepsilon=0$, $\D\in\mathfrak{der}(\ai)$, and $\delta\varphi_{\D}=\delta\varphi=0$.
\end{corollary}

\begin{remark}
	If $\rel e_0$ is central, then one can easily see that $(\g,\cro)$ is a sequence of two extensions: a central extension followed by a semi-direct product.
\end{remark}

Now, we will describe the relation between the Killing form of a symplectic Lie algebra $(\ai,[\,\cdot\,,\cdot\,]_\ai,\w_\ai)$ and the Killing form of its symplectic double extension by a line $(\g,[\,\cdot\,,\cdot\,],\w)$.

\begin{proposition}\label{KillingOfAdim1}
Let $(a^\ai,b^\ai,\D,\varepsilon)$ be an admissible element of $(\ai,\cro_\ai,\w_\ai)$, and $(\g,[\,\cdot\,,\cdot\,],\w)$ its symplectic double extension by a line. For each $a,b\in\ai$, we have:\vspace{1mm}
	\begin{enumerate}
		\item[$1.$] $\tr(\ad_a)=\tr(\ad^\ai_a)\,;$\vspace{2mm}
		\item[$2.$] $\tr(\ad_{d_0})=\tr(\D)+\varepsilon$,\, and  $\,\,\tr(\ad_{e_0})=0\,;$\vspace{2mm}
		\item[$3.$] $\kappa_\g(a,b)=\kappa_\ai(a,b)+2\,\w_\ai(a^\ai,a)\,\w_\ai(a^\ai,b)\,;$\vspace{2mm}
		\item[$4.$] $\kappa_\g(a,e_0)=\kappa_\g(d_0,e_0)=\kappa_\g(e_0,e_0)=0\,;$\vspace{2mm}
		\item[$5.$] $\kappa_\g(d_0,d_0)=\tr(\D^2)+\varepsilon^2$,\, and $\,\,\kappa_\g(a,d_0)=\tr(\D\ro\ad^\ai_a)$.
	\end{enumerate}
\end{proposition}

\begin{proof}
	Let us prove the first equality, since for the others the proof is similar. If we denote by $\mathcal{M}_{\mathbb{B}_\ai}(\ad^\ai_a)$ the matrix representation of the linear endomorphism $\ad^\ai_a\in\gl(\ai)$ with respect to a basis $\mathbb{B}_\ai:=\{a_k\}_{1\leqslant k\leqslant 2r}$ of $\ai$, then the matrix representation of $\ad_a\in\gl(\g)$ in the basis $\mathbb{B}:=\{d_0\}\cup\mathbb{B}_\ai\cup\{e_0\}$ is:
	\begin{equation*}
		\mathcal{M}_\mathbb{B}(\ad_a)=\begin{pmatrix}
			-\w_\ai(a^\ai,a) & 0 & \cdots & 0 & 0\\
			-\D(a)& 	&\mathcal{M}_{\mathbb{B}_\ai}(\ad^\ai_a) &&\vdots\\
			& 	&&&0\\
			-\w_\ai(b^\ai,a) & \varphi(a,a_1) 	&\cdots&\varphi(a,a_{2r})&\w_\ai(a^\ai,a)
		\end{pmatrix}.
	\end{equation*}
	In particular, $\tr(\ad_a)=\tr(\ad^\ai_a)$.
\end{proof}

The next corollary is a direct consequence of the previous proposition.

\begin{corollary}\label{CorSec2-1UnimSolv}
	The following assertions hold:
	\begin{itemize}
		\item[$1.$] The symplectic Lie algebra $(\g,[\,\cdot\,,\cdot\,],\w)$ is unimodular if and only if $(\ai,[\,\cdot\,,\cdot\,]_\ai,\w_\ai)$ is unimodular and $\varepsilon=-\tr(\D)\,;$
		\item[$2.$] The symplectic Lie algebra $(\g,[\,\cdot\,,\cdot\,],\w)$ is solvable if and only if $(\ai,[\,\cdot\,,\cdot\,]_\ai,\w_\ai)$ is solvable and $\tr(\D\ro\ad^\ai_{[a,b]_\ai})=0$ for any $a,b\in\ai$.
	\end{itemize}
\end{corollary}

A solvable Lie algebra $(\g,\cro)$ is called \textit{split-solvable} (or \textit{completely solvable}) if it admits a complete flag of ideals, or equivalently (cf. \cite[p. 45]{Kna}), if the eigenvalues of all $\ad_x\in\gl(\g)$ are in $\rel$. It is clear, using Lie's theorem, that any split-solvable Lie algebra has an ideal of dimension $1$. The following lemma shows that if $(\g,\cro,\w)$ is split-solvable, then its symplectic reduction $(\ai,\cro_\ai,\w_\ai)$ with respect to a $1$-dimensional isotropic ideal $\rel e_0$ is also split-solvable.

\begin{lemma}
	If $(\g,\cro,\w)$ is split-solvable, then $(\ai,\cro_\ai,\w_\ai)$ is also split-solvable.
\end{lemma}

\begin{proof}
Let $a\in\ai$, and let $\mu$ be a nonzero eigenvalue of the linear endomorphism $\ad^\ai_a:\ai\to\ai$, i.e., there exists $b\in\ai\backslash\{0\}$ such that $[a,b]_\ai=\mu b$. If $\w_\ai(a^\ai,a)= 0$, then by \eqref{LiBraDobExIdeDim1Prop0}, we have:
\begin{equation*}
\ad_a\big([a,b]\big)=\big[a,[a,b]_\ai\big]=\mu [a,b].
\end{equation*}
Thus, $\mu$ is an eigenvalue of $\ad_a \in \gl(\g)$. Since $(\g, \cro)$ is split-solvable, it follows that $\mu \in \rel$. On the other hand, if $\w_\ai(a^\ai,a)\neq 0$, then either $\mu=\w_\ai(a^\ai,a)$ or $\mu\neq\w_\ai(a^\ai,a)$. In the first case, $\mu$ is obviously real. In the second case, i.e., if $\mu\neq\w_\ai(a^\ai,a)$, then put $\xi:=\frac{\varphi(a,b)}{\mu-\w_\ai(a^\ai,a)}$, and compute:
\begin{equation*}
\ad_a\big(b+\xi e_0\big)=[a,b]_\ai+\varphi(a,b)e_0+\xi\w_\ai(a^\ai,a) e_0=\mu(b+\xi e_0).
\end{equation*}
As a result, since $(\g,\cro)$ is split-solvable, we deduce that $\mu\in\rel$, which finishes the proof.
\end{proof}

As consequence, we obtain that:

\begin{proposition}
Any split-solvable symplectic Lie algebra $(\g,\cro,\w)$ is a symplectic double extension by a line of a split-solvable symplectic Lie algebra $(\ai,\cro_\ai,\w_\ai)$.
\end{proposition}

The following corollary is a generalization of \cite[Theorem $3.1$, p. 220]{DM}
\begin{corollary}
All split-solvable symplectic Lie algebras are obtained by a finite sequence of symplectic double extensions by a line, starting from the zero symplectic Lie algebra.
\end{corollary}

We end this section by providing some constructive examples. It is important to note that all the examples constructed are solvable symplectic Lie algebras with zero center.

\begin{example}
	Let $\ai:=\Span_\rel\{a_1,a_2\}$ be the $2$-dimensional abelian Lie algebra, and $\w_\ai:=a_1^*\wedge a_2^*$ a symplectic form on it. Consider the linear endomorphism $\J:\ai\to\ai$ of $\ai$ which is defined by the following matrix $\begin{psmallmatrix}
		0&-1\\
		1&0
	\end{psmallmatrix}$ in the basis $\{a_1,a_2\}$. It is obvious that $\J\in\mathfrak{der}(\ai)$. Therefore, if we take $(a^\ai,b^\ai,\D,\varepsilon)=(0,0,\J,1)$, then one can easily check that $(0,0,\J,1)$ is admissible. Hence, $\g:=\rel d_0\oplus\ai\oplus\rel e_0$ endowed with the following bracket
	\begin{equation*}
		\left\{
		\begin{array}{r c l}\vspace{1mm}
			\left[e_0,d_0\right]&=&e_0,\\
			\left[d_0,a\right]&=&\J(a),
		\end{array}
		\right.\\
	\end{equation*}
	where the other brackets are either zero or given by skew-symmetry, is a Lie algebra. Moreover, $\w:=\w_\ai+d_0^*\wedge e_0^*$ is a symplectic form on it.
\end{example}

\begin{example}
	Let $\ai:=\Span_\rel\{a_1,a_2\}$ be the $2$-dimensional abelian Lie algebra, $\mu\in\rel\backslash\{0\}$, $\w_{\ai,\mu}:=\tfrac{1}{2\mu}\,a_1^*\wedge a_2^*$ a symplectic form on it, and $\ef_\mu:\ai\to\ai$ denotes the linear endomorphism of $\ai$ which is defined by the matrix $\begin{psmallmatrix}
		\mu&1\\
		-1&\mu
	\end{psmallmatrix}$ in the basis $\{a_1,a_2\}$. By taking $(a^\ai,b^\ai,\D,\varepsilon)=(0,0,\ef_\mu,-2\mu)$, one can easily verify that it is indeed admissible. Thus, the vector space $\g_\mu:=\rel d_0\oplus\ai\oplus\rel e_0$ endowed with the following bracket
	\begin{equation*}
		\left\{
		\begin{array}{r c l}\vspace{1mm}
			\left[d_0,e_0\right]_\mu&=&2\mu\,e_0,\\\vspace{1mm}
			\left[d_0,a\right]_\mu&=&\ef_\mu(a),\\
			\left[a,b\right]_\mu&=&\big(a_1^*\wedge a_2^*\big)(a,b)\,e_0,
		\end{array}
		\right.\\
	\end{equation*}
	where the other brackets are either zero or given by skew-symmetry, is a Lie algebra. Further, $\w_\mu:=\w_{\ai,\mu}+d_0^*\wedge e_0^*$ is a symplectic form on it.
\end{example}

\begin{example}
	Let $\mu$ be a positive real number, and let $\ai_\mu:=\Span_\rel\{a_1,a_2,a_3,a_4,a_5,a_6\}$ be the $6$-dimensional Lie algebra defined by: 
	\begin{equation*}
		[a_1,a_3]_{\ai_\mu}=-a_3,\qquad [a_1,a_4]_{\ai_\mu}=a_4,\qquad [a_1,a_5]_{\ai_\mu}=\mu a_6,\et[a_1,a_6]_{\ai_\mu}=-\mu a_5,
	\end{equation*} 
	where the other brackets are either zero or given by skew-symmetry. Consider the following symplectic form on it $\w_{\ai_\mu}:=a_1^*\wedge a_2^*+a_3^*\wedge a_4^*+a_5^*\wedge a_6^*$, and let $\ef:\ai\to\ai$ be the linear endomorphism defined by $\ef(a_2)=a_3$, and $\ef(a_k)=0$ for $k\neq 2$. Clearly that $[\ai_\mu,\ai_\mu]^\perp=\Span_\rel\{a_1,a_2\}$, and therefore for $f:=-i_{a_2}\w_{\ai_\mu}=a_1^*$, one can easily check that $f\in\mathcal{Z}^1(\ai,\rel)$, and $\ef\in\mathcal{Z}^1(\ai,\widetilde{\rho}_{_f})$. Moreover, a direct computation yields that:
	\begin{eqnarray*}
		\varphi(a,b)&:=&\w_\ai(\ef(a),b)+\w_\ai(a,\ef(b))\,=\,\left(a_2^*\wedge a_4^*\right)(a,b),\\
		\varphi_{\ef}(a,b)&:=&\varphi(\ef(a),b)+\varphi(a,\ef(b))\,=\,0.
	\end{eqnarray*}
	In particular, one can verify that $\delta\varphi=-f\wedge\varphi$. Hence, if we take $(a^\ai,b^\ai,\D,\varepsilon)=(-a_2,0,\ef,0)$, one can easily see that it is an admissible element. As a result, $\g_\mu:=\rel d_0\oplus\ai_\mu\oplus\rel e_0$ endowed with the following bracket
	\begin{equation*}
		\left\{
		\begin{array}{r c l}\vspace{1mm}
			\left[d_0,a\right]_\mu&=&a_1^*(a)d_0+\ef(a),\\\vspace{1mm}
			\left[a,e_0\right]_\mu&=&a_1^*(a)e_0,\\
			\left[a,b\right]_\mu&=&[a,b]_{\ai_\mu}+\big(a_2^*\wedge a_4^*\big)(a,b)\,e_0,
			
		\end{array}
		\right.\\
	\end{equation*}
	where the other brackets are either zero or given by skew-symmetry, is a Lie algebra. Furthermore, $\w_\mu:=\w_{\ai_\mu}+d_0^*\wedge e_0^*$ is a symplectic form on $(\g_\mu,\cro_\mu)$.
\end{example}

\section{Symplectic double extension by a plane}\label{Section3}
In this section, we introduce the symplectic double extension by a plane for symplectic Lie algebras. As in the previous section, given a symplectic Lie algebra $(\ai,\cro_\ai,\w_\ai)$, we construct a symplectic Lie algebra $(\g:=\rel d_0\oplus\rel d_1\oplus\ai\oplus\rel e_0\oplus\rel e_1,\cro,\w)$, where $\{d_0,e_0,d_1,e_1\}$ is a basis of $\rel^4$ such that $(\ai,\cro_\ai,\w_\ai)$ is a symplectic Lie subalgebra of it. This construction involves two successive Lie algebra extensions:
\begin{enumerate}
	\item A generalized semi-direct product (in the sense of Remark \ref{Rmq1GenSemiDirProd}) of the $2$-dimensional abelian Lie algebra $\rel e_0\oplus\rel e_1$ by the Lie algebra $(\ai,\cro_\ai)$. 
	\item A generalized bicrossed product (in the sense of Remark \ref{Rmq1GenBicrProd}) of the $2$-dimensional Lie algebra $\rel d_0 \oplus \rel d_1$, endowed with a specific Lie bracket, and the Lie algebra $\ai \oplus \rel e_0 \oplus \rel e_1$, which arises from the first extension. 
\end{enumerate}
The resulting symplectic Lie algebra $(\g, \cro,\w)$ is thus called the \textit{symplectic double extension} of $(\ai,\cro_\ai,\w_\ai)$ \textit{by a plane}.

More precisely, let $(\ai,[\,\cdot\,,\cdot\,]_\ai,\w_\ai)$ be a symplectic Lie algebra, $a^\ai_0,a^\ai_1,b^\ai_0,b^\ai_1,b_2^\ai,z^\ai\in\ai$, $\D_0,\D_1:\ai\to\ai$ two linear endomorphisms of $\ai$, and $\varepsilon_k\in\rel$, for $k\in\{0,\ldots,5\}$. Similarly to the previous section, we endow the vector space $\g:=\rel d_0\oplus\rel d_1\oplus\ai\oplus\rel e_0\oplus\rel e_1$, where $\{e_0,d_0,e_1,d_1\}$ is a basis of $\rel^4$, with the bracket $\cro\in\bigwedge^2\g^*\otimes\g$ given by:
\begin{equation*}\label{LiBraDobExIdeDim2Prop0}
	\left\{
	\begin{array}{r c l}\vspace{1mm}
		\left[d_0,e_0\right]&=&\varepsilon_0e_0+\varepsilon_1e_1,\\\vspace{1mm}
		\left[d_0,e_1\right]&=&-\varepsilon_1e_0+\varepsilon_0e_1,\\\vspace{1mm}
		\left[d_1,e_0\right]&=&\varepsilon_2e_0+\varepsilon_3e_1,\\\vspace{1mm}
		\left[d_1,e_1\right]&=&-\varepsilon_3e_0+\varepsilon_2e_1,\\\vspace{1mm}
		\left[d_0,d_1\right]&=&(\varepsilon_2-\varepsilon_1)d_0-(\varepsilon_0+\varepsilon_3)d_1+z^\ai+\varepsilon_4e_0+\varepsilon_5e_1,\\\vspace{1mm}
		\left[d_0,a\right]&=&\w_\ai(a^\ai_0,a)d_0+\w_\ai(a^\ai_1,a)d_1+\D_0(a)+\w_\ai(b^\ai_0,a)e_0+\w_\ai(b^\ai_1,a)e_1,\\\vspace{1mm}
		\left[d_1,a\right]&=&-\w_\ai(a^\ai_1,a)d_0+\w_\ai(a^\ai_0,a)d_1+\D_1(a)+\w_\ai(b^\ai_1+z^\ai,a)e_0+\w_\ai(b^\ai_2,a)e_1,\\\vspace{1mm}
		\left[e_0,a\right]&=&-\w_\ai(a^\ai_0,a)e_0+\w_\ai(a^\ai_1,a)e_1,\hfill\\\vspace{1mm}
		\left[e_1,a\right]&=&-\w_\ai(a^\ai_1,a)e_0-\w_\ai(a^\ai_0,a)e_1,\hfill\\
		\left[a,b\right]&=&[a,b]_\ai+\w_\ai\big(\D_0(a),b\big)e_0+\w_\ai\big(a,\D_0(b)\big)e_0+\w_\ai\big(\D_1(a),b\big)e_1+\w_\ai\big(a,\D_1(b)\big)e_1,\hfill
	\end{array}
	\right.\\
\end{equation*}
where the unspecified brackets are either zero or given by skew-symmetry. Moreover, we  extend $\w_\ai$ to a nondegenerate skew-symmetric bilinear form $\w\in\bigwedge^2\g^*$ on $\g$ as follows: $$\w:=\w_\ai+d_0^*\wedge e_0^*+d_1^*\wedge e_1^*.$$
Thus, we have:

\begin{proposition}\label{PropSymLiAlgDoblExtDim2Prop0}
$(\g,\cro,\w)$ is a symplectic Lie algebra if and only if
\begin{itemize}
	\item[$1.$] $a^\ai_0,a^\ai_1\in[\ai,\ai]^\perp$, and
	\begin{equation}\label{PropSymLiAlgDoblExtDim2Prop0-C2}
\left\{
\begin{array}{r c l}\vspace{1mm}
\D^\divideontimes_0(a^\ai_0)&=&-\varepsilon_0a^\ai_0-\varepsilon_2a^\ai_1,\\	\vspace{1mm}
\D^\divideontimes_0(a^\ai_1)&=&\varepsilon_1a^\ai_0+\varepsilon_3a^\ai_1,\\\vspace{1mm}
\D^\divideontimes_1(a^\ai_0)&=&-\varepsilon_2a^\ai_0+\varepsilon_0a^\ai_1,\\
\D^\divideontimes_1(a^\ai_1)&=&\varepsilon_3a^\ai_0-\varepsilon_1a^\ai_1\,;
\end{array}	
		\right.\\\vspace{1mm}
	\end{equation}
	\item[$2.$] $\varepsilon_0\varepsilon_1+\varepsilon_2\varepsilon_3=\w_\ai(a^\ai_0,z^\ai)$,\, and\, $\varepsilon_1(\varepsilon_2-\varepsilon_1)-\varepsilon_3(\varepsilon_0+\varepsilon_3)=\w_\ai(a^\ai_1,z^\ai)\,;$\vspace{1mm}
	\item[$3.$] The two linear endomorphisms $\D_0$ and $\D_1$ of $\ai$ satisfy:
	\begin{equation*}
		\widetilde{\delta}_{_{a^\ai_0}}\D_0=\big\llbracket i_{a^\ai_1}\w_\ai,\D_1\big\rrbracket,\et \widetilde{\delta}_{_{a^\ai_0}}\D_1=-\big\llbracket i_{a^\ai_1}\w_\ai,\D_0\big\rrbracket,
	\end{equation*}
	where $\widetilde{\delta}_{_{a^\ai_0}}$ is the Chevalley-Eilenberg differential corresponding to the following representation:
	\begin{equation*}
		\widetilde{\rho}_{_{a^\ai_0}}:\ai\longrightarrow\gl(\ai),\qquad a\longmapsto\ad^\ai_a+\w_\ai(a^\ai_0,a)\id_\ai\,;
	\end{equation*}
	\item[$4.$] For $k\in\{0,1\}$, the skew-symmetric bilinear forms $\psi_k:\ai\times\ai\to\rel$ defined by:
	\begin{equation*}
		\varphi_k(a,b):=\w_\ai\big(\D_k(a),b\big)+\w_\ai\big(a,\D_k(b)\big),
	\end{equation*}
with $a,b\in\ai$, satisfy the following:
	\begin{equation*}
		\delta_{_{a^\ai_0}}\varphi_0=\varphi_1\wedge (i_{a^\ai_1}\w_\ai),\et\delta_{_{a^\ai_0}}\varphi_1=-\varphi_0\wedge (i_{a^\ai_1}\w_\ai),
	\end{equation*}
	where $\delta_{_{a^\ai_0}}$ is the Chevalley-Eilenberg differential corresponding to the representation: 
	\begin{equation*}
		\rho_{_{a^\ai_0}}:\ai\longrightarrow\gl(\rel),\qquad a\longmapsto \w_\ai(a^\ai_0,a)\id_\rel\,;
	\end{equation*}
	\item[$5.$] The following equalities are satisfied
	\begin{equation}\label{PropSymLiAlgDoblExtDim2Prop0-C3}
		\left\{
		\begin{array}{r c l}\vspace{0.5mm}
			i_{\D_0^\divideontimes(b^\ai_1+z^\ai)}\w_\ai-i_{\D_1^\divideontimes(b^\ai_0)}\w_\ai+i_{z^\ai}\varphi_0&=&\big(\varepsilon_1-2\varepsilon_2\big)i_{b_0^\ai}\w_\ai+\big(2\varepsilon_0+\varepsilon_3\big)i_{b_1^\ai+z^\ai}\w_\ai-\varepsilon_1i_{b_2^\ai}\w_\ai,\\\vspace{2mm}
			&&+\,\varepsilon_3i_{b_1^\ai}\w_\ai+3\varepsilon_4 i_{a_0^\ai}\w_\ai+\varepsilon_5 i_{a_1^\ai}\w_\ai,\vspace{0.5mm}\\
i_{\D_0^\divideontimes(b_2^\ai)}\w_\ai-i_{\D_1^\divideontimes(b_1^\ai)}\w_\ai+i_{z^\ai}\varphi_1&=&\big(\varepsilon_1-2\varepsilon_2\big)i_{b_1^\ai}\w_\ai+\big(2\varepsilon_0+\varepsilon_3\big)i_{b_2^\ai}\w_\ai+\varepsilon_1i_{b_1^\ai+z^\ai}\w_\ai\\\vspace{2mm}
&&-\,\varepsilon_3i_{b^\ai_0}\w_\ai-\varepsilon_4i_{a^\ai_1}\w_\ai+3\varepsilon_5i_{a^\ai_0}\w_\ai,\\
\left[\D_0,\D_1\right]+2\,z^\ai\cdot i_{a^\ai_0}\w_\ai&=&\ad^\ai_{z^\ai}+(\varepsilon_2-\varepsilon_1)\D_0-(\varepsilon_0+\varepsilon_3)\D_1,
		\end{array}	\qquad\qquad
		\right.
	\end{equation}
	
where $z^\ai\cdot i_{a^\ai_0}\w_\ai\in\gl(\ai)\,$ given by $\,\big(z^\ai\cdot i_{a^\ai_0}\w_\ai\big)(a):=\w_\ai(a^\ai_0,a)z^\ai\,;$
	\item[$6.$] The following equalities are fulfilled
	\begin{equation}\label{PropSymLiAlgDoblExtDim2Prop0-C4}
		\left\{
		\begin{array}{r c l}\vspace{2mm}
			\varphi^0_{\D_{0}}-\varepsilon_0\varphi_0+\varepsilon_1\varphi_1&=&-\delta(i_{b^\ai_0}\w_\ai)+2(i_{b^\ai_0}\w_\ai)\wedge (i_{a^\ai_0}\w_\ai)+(i_{2b^\ai_1+z^\ai}\w_\ai)\wedge (i_{a^\ai_1}\w_\ai),\\\vspace{2mm}
			\varphi^1_{\D_{0}}-\varepsilon_0\varphi_1-\varepsilon_1\varphi_0&=&-\delta(i_{b^\ai_1}\w_\ai)+2(i_{b^\ai_1}\w_\ai)\wedge (i_{a^\ai_0}\w_\ai)+(i_{b^\ai_2-b^\ai_0}\w_\ai)\wedge (i_{a^\ai_1}\w_\ai),\\\vspace{2mm}
			\varphi^0_{\D_{1}}-\varepsilon_2\varphi_0+\varepsilon_3\varphi_1&=&-\delta(i_{b^\ai_1+z^\ai}\w_\ai)+2(i_{b^\ai_1+z^\ai}\w_\ai)\wedge (i_{a^\ai_0}\w_\ai)+(i_{b^\ai_2-b^\ai_0}\w_\ai)\wedge (i_{a^\ai_1}\w_\ai),\\
			\varphi^1_{\D_{1}}-\varepsilon_2\varphi_1-\varepsilon_3\varphi_0&=&-\delta(i_{b^\ai_2}\w_\ai)+2(i_{b^\ai_2}\w_\ai)\wedge (i_{a^\ai_0}\w_\ai)+(i_{b^\ai_2-2b^\ai_1-z^\ai}\w_\ai)\wedge (i_{a^\ai_1}\w_\ai),\qquad\qquad
		\end{array}	
		\right.
	\end{equation}
	
where $\varphi^k_{\D_{k'}}:\ai\times\ai\to\rel$ is the skew-symmetric bilinear form given by:
	\begin{equation*}
		\varphi^k_{\D_{k'}}(a,b):=\varphi_k\big(\D_{k'}(a),b\big)+\varphi_k\big(a,\D_{k'}(b)\big),\qquad\forall\,a,b\in\ai,\,\,\forall\,k,k'\in\{0,1\}.
	\end{equation*}
\end{itemize}
\end{proposition}

\begin{proof}
First, we endow the vector space $\bi:=\rel d_0\oplus\rel d_1$ with the following Lie bracket:
\begin{equation*}
[d_0,d_1]_\bi:=(\varepsilon_2-\varepsilon_1)d_0-(\varepsilon_0+\varepsilon_3)d_1.
\end{equation*}
Let $\bi^*:=\rel e_0\oplus\rel e_1$ be its dual, and assume, without loss of generality, that $\{e_0,e_1\}$ is the dual basis of $\{d_0,d_1\}$. Consider the linear maps $\rho_1:\ai\to\gl(\bi^*),\,\rho_2:\bi\to\gl(\bi^*),\,\pi_1:\ai\to\gl(\bi),\,\pi_2:\bi\to\gl(\ai),$ defined by:
	\begin{equation*}
	\left\{
	\begin{array}{r c l}\vspace{1mm}
\rho_2\big(s_0d_0+s_1d_1\big)\big(t_0e_0+t_1e_1\big)&:=&\Big(s_0\big(t_0\varepsilon_0-t_1\varepsilon_1\big)+s_1\big(t_0\varepsilon_2-t_1\varepsilon_3\big)\Big)e_0+\Big(s_0\big(t_0\varepsilon_1+t_1\varepsilon_0\big)+s_1\big(t_0\varepsilon_3+t_1\varepsilon_2\big)\Big)e_1,\\\vspace{1mm}
\pi_1(a)\big(s_0d_0+s_1d_1\big)&:=&\left(\w_\ai(s_0a,a_0^\ai)-\w_\ai(s_1a,a_1^\ai)\right)d_0+\left(\w_\ai(s_0a,a_1^\ai)+\w_\ai(s_1a,a_0^\ai)\right)d_1,\\\vspace{1mm}
\pi_2\big(s_0d_0+s_1d_1\big)(a)&:=&s_0\D_0(a)+s_1\D_1(a),\\\vspace{1mm}
\rho_1(a)\big(t_0e_0+t_1e_1\big)&:=&\left(\w_\ai(a_0^\ai,t_0a)+\w_\ai(a_1^\ai,t_1a)\right)e_0+\left(\w_\ai(a_0^\ai,t_1a)-\w_\ai(a_1^\ai,t_0a)\right)e_1,
	\end{array}	
	\right.\\
\end{equation*}
for $a\in\ai$ and $s_0,s_1,t_0,t_1\in\rel$. Furthermore, consider the skew-symmetric bilinear maps $\Phi_1:\bi\times\bi\to\ai$, $\Phi_2:\bi\times\bi\to\bi^*$, $\psi:\ai\times\ai\to\bi^*$, given by:
	\begin{equation*}
	\left\{
	\begin{array}{r c l}\vspace{1mm}
		\Phi_1\Big(s_0d_0+s_1d_1,s'_0d_0+s'_1d_1\Big)&:=&\big(s_0s'_1-s_0's_1\big)z^\ai,\\\vspace{1mm}
		\Phi_2\Big(s_0d_0+s_1d_1,s'_0d_0+s'_1d_1\Big)&:=&\varepsilon_4\big(s_0s'_1-s_0's_1\big)e_0+\varepsilon_5\big(s_0s'_1-s_0's_1\big)e_1,\\\vspace{1mm}
		\psi(a,b)&:=&\varphi_0(a,b)e_0+\varphi_1(a,b)e_1,\\
	\end{array}	
	\right.\\
\end{equation*}
where $a,b\in\ai$, $s_0,s_0',s_1,s_1'\in\rel$, and the bilinear map $\Omega:\bi\times\ai\to\bi^*$ defined by:
\begin{equation*}
\Omega\big(s_0d_0+s_1d_1,a\big):=\Big(\w_\ai(b_0^\ai,s_0a)+\w_\ai\big(b_1^\ai+z^\ai,s_1a\big)\Big)e_0+\Big(\w_\ai(b_1^\ai,s_0a)+\w_\ai\big(b_2^\ai,s_1a\big)\Big)e_1.
\end{equation*}
One can easily see that all five conditions of Proposition \ref{OmegaIsSymplecticGMP}, as well as the second condition of Proposition \ref{ConditionsPropoOfDoublSymplExten}, are satisfied. Moreover, the linear map $\rho_1:\ai \to \gl(\bi^*)$ is a representation if and only if $a_0^\ai, a_1^\ai \in [\ai, \ai]^\perp$, and the condition $\psi \in \mathcal{Z}^2(\ai, \rho_1)$ is equivalent to the fourth condition. Furthermore, the first condition of Proposition \ref{ConditionsPropoOfDoublSymplExten} is equivalent to \eqref{PropSymLiAlgDoblExtDim2Prop0-C2}, while the third condition is equivalent to \eqref{PropSymLiAlgDoblExtDim2Prop0-C3}. In addition, the fourth condition of Proposition \ref{ConditionsPropoOfDoublSymplExten} is equivalent to the second, and finally, the fifth condition corresponds to \eqref{PropSymLiAlgDoblExtDim2Prop0-C4}.
\end{proof}

\begin{remark}
\begin{enumerate}
\item If $a^\ai_0=a^\ai_1=0$, then the two linear endomorphisms $\D_0,\D_1\in\gl(\ai)$ are derivations of the Lie algebra $(\ai,\cro_\ai)$, and $\delta\varphi_{\D_{k'}}^k=\delta\varphi_k=0$, for $k,k'\in\{0,1\}$.
\item The abelian ideal $\rel e_0\oplus\rel e_1$ of $(\g,\cro,\w)$ is normal if and only if $a^\ai_0=a^\ai_1=0.$ In this case, the first extension is a central extension of the Lie algebra $\ai$ by means of $\varphi_0,\varphi_1\in\mathcal{Z}^2(\ai,\rel)$, and the second extension is the generalized bicrossed product of this central extension and the $2$-dimensional Lie algebra $(\bi:=\rel d_0\oplus\rel d_1,\cro_\bi)$.
\end{enumerate}
\end{remark}

\begin{definition}\label{AdmiSDIm2}
	An element $(a^\ai_0,a^\ai_1,b^\ai_0,b^\ai_1,b^\ai_2,z^\ai,\D_0,\D_1,\varepsilon_0,\ldots,\varepsilon_5)\in\ai^6\times\gl(\ai)^2\times\rel^6$ satisfying the sixth conditions of Proposition \ref{PropSymLiAlgDoblExtDim2Prop0} is called \textit{admissible} or an \textit{admissible element of} $(\ai,\cro_\ai,\w_\ai)$.
\end{definition}

\begin{definition}
	The symplectic Lie algebra $(\g,\cro,\w)$ obtained in Proposition \ref{PropSymLiAlgDoblExtDim2Prop0} is called the \textit{symplectic double extension of} $(\ai,\cro_\ai,\w_\ai)$ \textit{by a plane}.
\end{definition}

The next proposition gives a necessarily and sufficient condition for the symplectic double extension of $(\ai,\cro_\ai,\w_\ai)$ by a plane to be unimodular.

\begin{proposition}\label{PropSymLiAlgDoblExtDim2Prop01}
	The symplectic Lie algebra $(\g,[\,\cdot\,,\cdot\,],\w)$ is unimodular if and only if $(\ai,[\,\cdot\,,\cdot\,]_\ai,\w_\ai)$ is unimodular, $\varepsilon_1-\varepsilon_2=\tr(\D_1)$, and $\varepsilon_3-\varepsilon_0=\tr(\D_0)$. 
\end{proposition}

\begin{proof}
	A direct computation yields that:
	\begin{equation*}
		\left\{
		\begin{array}{r c l}\vspace{1mm}
			\tr(\ad_{d_0})&=&\tr(\D_0)+\varepsilon_0-\varepsilon_3,\\\vspace{1mm}
			\tr(\ad_{d_1})&=&\tr(\D_1)+\varepsilon_2-\varepsilon_1,\\\vspace{1mm}
			\tr(\ad_{e_0})&=&0,\\\vspace{1mm}
			\tr(\ad_{e_1})&=&0,\\
			\tr(\ad_a)&=&\tr(\ad_a^\ai),
		\end{array}
		\right.\\
	\end{equation*}
for any $a\in\ai$. Hence, the result follows.
\end{proof}

We conclude this section by the following illustrative example.

\begin{example}
	Let $\mu$ be a positive real number, and let $\ai_\mu:=\Span_\rel\{a_1,a_2,a_3,a_4\}$ be the $4$-dimensional Lie algebra defined by: 
	\begin{equation*}
		[a_1,a_3]_{\ai_\mu}=\mu a_4,\qquad [a_1,a_4]_{\ai_\mu}=-\mu a_3,
	\end{equation*} 
	where the other brackets are either zero or given by skew-symmetry. Consider the following symplectic form on it $\w_{\ai_\mu}:=a_1^*\wedge a_2^*+a_3^*\wedge a_4^*$. It is easy to check that $[\ai_\mu,\ai_\mu]^\perp=\Span_\rel\{a_1,a_2\}$. Thus, if we take $a^\ai_0:=-a_2,a^\ai_1:=a_1$, and all the other parameters to be zero, one can easily see that $(-a_2,a_1,0,\ldots,0)$ is admissible. As a result, it follows that $\g_\mu:=\rel d_0\oplus\rel d_1\oplus\ai_\mu\oplus\rel e_0\oplus\rel e_1$ endowed with the following bracket
	\begin{equation*}
		\left\{
		\begin{array}{r c l}\vspace{1mm}
\left[d_0,a\right]_\mu&=&a_1^*(a)d_0+a_2^*(a)d_1,\\\vspace{1mm}
		\left[d_1,a\right]_\mu&=&-a_2^*(a)d_0+a_1^*(a)d_1,\\\vspace{1mm}			\left[e_0,a\right]_\mu&=&-a_1^*(a)e_0+a_2^*(a)e_1,\\\vspace{1mm}
	\left[e_1,a\right]_\mu&=&-a_2^*(a)e_0-a_1^*(a)e_1,\\
	\left[a,b\right]_\mu&=&[a,b]_{\ai_\mu},
		\end{array}
		\right.\\
	\end{equation*}
	where the other brackets are either zero or given by skew-symmetry, is a Lie algebra. Further, $\w_\mu:=\w_{\ai_\mu}+d_0^*\wedge e_0^*+d_1^*\wedge e_1^*$ is a symplectic form on it. Note that, $(\g_\mu,\cro_\mu)$ is $2$-step solvable, and $\mathfrak{z}(\g_\mu)=\{0\}$.
\end{example}

\section{Symplectic double extension for solvable symplectic Lie algebras}\label{Section3p}

This section, which forms the core of this paper, provides a complete and detailed description of the symplectic double extension process for arbitrary solvable symplectic Lie algebras.

Given a solvable symplectic Lie algebra $(\g,[\,\cdot\,,\cdot\,],\w)$, two distinct cases arise: either $(\g,[\,\cdot\,,\cdot\,],\w)$ possesses a nonzero isotropic ideal, or it does not. Since the second case will be discussed in a broader context in Section \ref{Art4SectionIrreducible}, our focus here will be on the first case. Let $(\g,[\,\cdot\,,\cdot\,],\w)$ be a solvable symplectic Lie algebra which admits a nonzero isotropic ideal $\jj\subset\g$, and consider the following representation:
\begin{equation*}
\pi:\g\longrightarrow\gl(\jj),\quad\textnormal{written}\quad x\longmapsto\ad_{x|\jj}.
\end{equation*}
Then, by Lie's theorem (cf. \cite[p. 46]{Bour}), there exists an isotropic ideal $\ii\subseteq\jj$ of $(\g,\cro,\w)$ such that $\dim\ii$ is either equal to $1$ or $2$. The first case, involving a $1$-dimensional ideal, was addressed in Section \ref{Ar4sect1} (see Theorem \ref{TheoDoubExtdim1}). So, in this section, we will focus on the latter case, where $(\g,\cro,\w)$ contains a $2$-dimensional isotropic ideal and no $1$-dimensional ideal.\vspace{2mm}

Let us first start with the following useful lemma.

\begin{lemma}\label{LemSolIrrDim2}
	Let $(\g,\cro)$ be a solvable Lie algebra, and $\pi:\g\to\gl(V)$ a $2$-dimensional irreducible representation. Then, there exists a basis $\{e_0,e_1\}$ of $V$, and $\alpha,\beta\in\g^*$ such that:
	\begin{equation}\label{idealsol2}
		\left\{
		\begin{array}{r c l}\vspace{1mm}
			\pi(x)(e_0)&=&\alpha(x)e_0-\beta(x)e_1,\\ 
			\pi(x)(e_1)&=&\beta(x)e_0+\alpha(x)e_1,
		\end{array}
		\right.\\
	\end{equation}
	for any $x\in\g$.
\end{lemma}

\begin{proof}
	Let $\pi_\cplx:\g_\cplx\to\gl(V_\cplx)$ denotes the complexification of the real representation $\pi:\g\to\gl(V)$. Since $(\g_\cplx,\cro_\cplx)$ is solvabe, there exists (Lie's theorem) a nonzero element $e_0+\mathbf{i}e_1\in V_\cplx\backslash\{0\}$  such that 
	\begin{equation*}
		\pi_\cplx(x+\mathbf{i}y)(e_0+\mathbf{i}e_1)=\lambda(x+\mathbf{i}y)e_0+\mathbf{i}\lambda(x+\mathbf{i}y)e_1,
	\end{equation*}
	for any $x,y\in\g$, where $\lambda:\g_\cplx\to\cplx$ is a complex linear form. In other words, if we write $\lambda=\widehat{\alpha}+i\widehat{\beta}$, with $\widehat{\alpha}:=\Re(\lambda)$, and $\widehat{\beta}:=\Im(\lambda)$, then 
	\begin{equation*}
		\left\{
		\begin{array}{r c l}\vspace{1mm}
			\pi(x)(e_0)-\pi(y)(e_1)&=&\Big(\widehat{\alpha}(x+\mathbf{i}y)-\widehat{\beta}(x+\mathbf{i}y)\Big)e_0-\Big(\widehat{\beta}(x+\mathbf{i}y)+\widehat{\alpha}(x+\mathbf{i}y)\Big)e_1,\\ 
			\pi(x)(e_1)+\pi(y)(e_0)&=&\Big(\widehat{\beta}(x+\mathbf{i}y)+\widehat{\alpha}(x+\mathbf{i}y)\Big)e_0+\Big(\widehat{\alpha}(x+\mathbf{i}y)-\widehat{\beta}(x+\mathbf{i}y)\Big)e_1,
		\end{array}
		\right.\\
	\end{equation*}
	for any $x,y\in\g$. Hence, by taking $y=x$, $\alpha(x):=\widehat{\alpha}(x+\mathbf{i}x)$, and $\beta(x):=\widehat{\beta}(x+\mathbf{i}x)$, we obtain \eqref{idealsol2}. Finally, if $e_0=0$ (resp. $e_1=0$), then one can easily see that $\rel e_1$ (resp. $\rel e_0$) is a $1$-dimensional $\g$-submodule of $V$, which is a contradiction. Moreover, if $\{e_0,e_1\}$ are linearly dependent, i.e., there exists $\xi\in\rel\backslash\{0\}$ such that $e_1=\xi e_0$, then one can check that $\rel e_0$ is a $1$-dimensional $\g$-submodule of $V$, which is impossible. As a result, $\{e_1,e_2\}$ must be a basis of $V$.
\end{proof}

Let $(\g,[\,\cdot\,,\cdot\,],\w)$ be a solvable symplectic Lie algebra, and assume that there exists a $2$-dimensional isotropic ideal $\mathfrak{i}$ of $(\g,\cro,\w)$, and $(\g,\cro)$ has no $1$-dimensional ideal. Since $\mathfrak{i}$ is an isotropic ideal of $\g$, one can easily check that it is abelian, i.e., $[\ii,\ii]=\{0\}$. Moreover, if we choose a basis $\{e_0,e_1\}$ of $\mathfrak{i}$ such as in Lemma \ref{LemSolIrrDim2}, then for $d_0,d_1\in\g\backslash\{0\}$ such that $\w(d_k,e_k)=1$, and $\w(d_k,d_{k'})=\w(d_k,e_{k'})=0$, we can decompose $\g$ as a vector space direct sum, i.e.,
\begin{equation*}
	\g=\rel d_0\oplus\rel d_1\oplus\ai\oplus\rel e_0\oplus\rel e_1,
\end{equation*}
where $\ai\cong\mathfrak{i}^\perp/\mathfrak{i}$, $\w_\ai:=\w_{|\ai\times\ai}$ is nondegenerate, and $\ai^\perp=\Span_\rel\{d_0,d_1,e_0,e_1\}$. For $a,b\in\ai$, and by using Lemma \ref{LemSolIrrDim2}, we can write:
\begin{equation}\label{eqSolDim2}
	\left\{
	\begin{array}{r c l}\vspace{1mm}
		\left[d_0,e_0\right]&=&\varepsilon_0e_0+\varepsilon_1e_1,\\\vspace{1mm}
		\left[d_0,e_1\right]&=&-\varepsilon_1e_0+\varepsilon_0e_1,\\\vspace{1mm}
		\left[d_1,e_0\right]&=&\varepsilon_2e_0+\varepsilon_3e_1,\\\vspace{1mm}
		\left[d_1,e_1\right]&=&-\varepsilon_3e_0+\varepsilon_2e_1,\\\vspace{1mm}
		\left[d_0,d_1\right]&=&\mu_1d_0+\mu_2d_1+z^\ai+\varepsilon_4e_0+\varepsilon_5e_1,\\\vspace{1mm}
		\left[d_0,a\right]&=&f_0(a)d_0+f_1(a)d_1+\D_0(a)+\ell_0(a)e_0+\ell_1(a)e_1,\\ \vspace{1mm}
		\left[d_1,a\right]&=&f_2(a)d_0+f_3(a)d_1+\D_1(a)+\ell_3(a)e_0+\ell_2(a)e_1,\\\vspace{1mm}
		\left[e_0,a\right]&=&\lambda_0(a)e_0+\lambda_1(a)e_1,\hfill\\\vspace{1mm}
		\left[e_1,a\right]&=&-\lambda_1(a)e_0+\lambda_0(a)e_1,\hfill\\
		\left[a,b\right]&=&[a,b]_\ai+\varphi_0(a,b)e_0+\varphi_1(a,b)e_1,
	\end{array}
	\right.\\
\end{equation}
where $f_0,f_1,f_2,f_3,\ell_0,\ell_1,\ell_2,\ell_3,\lambda_0,\lambda_1\in\ai^*,\,\D_0,\D_1\in\gl(\ai),\,[\,\cdot\,,\cdot\,]_\ai\in\bigwedge^2\ai^*\otimes\ai,\,\varphi_0,\varphi_1\in\bigwedge^2\ai^*,\,z^\ai\in\ai$, and $\mu_1,\mu_2,\varepsilon_0,\varepsilon_1,\varepsilon_2,\varepsilon_3,\varepsilon_4,\varepsilon_5\in\rel$. The following proposition provides all the consequences of $\w$ being a $2$-cocycle of $(\g,\cro)$.

\begin{proposition}
	The following identities hold:
	\begin{itemize}
		\item[$1.$] $f_0(a)=\w([d_0,a],e_0)=-\lambda_0(a)\,;$\vspace{1mm}
		\item[$2.$] $f_1(a)=\w([d_0,a],e_1)=\lambda_1(a)\,;$\vspace{1mm}
		\item[$3.$] $f_2(a)=\w([d_1,a],e_0)=-\lambda_1(a)\,;$\vspace{1mm}
		\item[$4.$] $f_3(a)=\w([d_1,a],e_1)=-\lambda_0(a)\,;$\vspace{1mm}
		\item[$5.$] $\ell_0(a)=\w(d_0,[d_0,a])\,;$\vspace{1mm}
		\item[$6.$] $\ell_2(a)=\w(d_1,[d_1,a])\,;$\vspace{1mm}
		\item[$7.$] $\ell_3(a)=\w(d_0,[d_1,a])=\ell_1(a)+\w_\ai(z^\ai,a)\,;$\vspace{1mm}
		\item[$8.$] $\w(\D_k(a),b)=\w([d_k,a],b)$,\, for $k\in\{0,1\}\,;$\vspace{1mm}
		\item[$9.$] $\varphi_k(a,b)=\w(d_k,[a,b])=\w_\ai\big(\D_k(a),b\big)+\w_\ai\big(a,\D_k(b)\big)$,\, for $k\in\{0,1\}\,;$\vspace{1mm}
		\item[$10.$] $\mu_1=\w([d_0,d_1],e_0)=\varepsilon_2-\varepsilon_1\,;$\vspace{1mm}
		\item[$11.$] $\mu_2=\w([d_0,d_1],e_1)=-\varepsilon_0-\varepsilon_3$.
	\end{itemize}
\end{proposition}

\begin{proof}
	A straightforward calculation, using the fact that $\w$ is a $2$-cocycle of $(\g,\cro)$.
\end{proof}

Hence, the equalities in \eqref{eqSolDim2} become:
\begin{equation}\label{eqSolDim2II}
	\left\{
	\begin{array}{r c l}\vspace{1mm}
		\left[d_0,e_0\right]&=&\varepsilon_0e_0+\varepsilon_1e_1,\\\vspace{1mm}
		\left[d_0,e_1\right]&=&-\varepsilon_1e_0+\varepsilon_0e_1,\\\vspace{1mm}
		\left[d_1,e_0\right]&=&\varepsilon_2e_0+\varepsilon_3e_1,\\\vspace{1mm}
		\left[d_1,e_1\right]&=&-\varepsilon_3e_0+\varepsilon_2e_1,\\\vspace{1mm}
		\left[d_0,d_1\right]&=&(\varepsilon_2-\varepsilon_1)d_0-(\varepsilon_0+\varepsilon_3)d_1+z^\ai+\varepsilon_4e_0+\varepsilon_5e_1,\\\vspace{1mm}
		\left[d_0,a\right]&=&f_0(a)d_0+f_1(a)d_1+\D_0(a)+\ell_0(a)e_0+\ell_1(a)e_1,\\\vspace{1mm} 
		\left[d_1,a\right]&=&-f_1(a)d_0+f_0(a)d_1+\D_1(a)+\big(\ell_1(a)+\w_\ai(z^\ai,a)\big)e_0+\ell_2(a)e_1,\\\vspace{1mm}
		\left[e_0,a\right]&=&-f_0(a)e_0+f_1(a)e_1,\hfill\\\vspace{1mm}
		\left[e_1,a\right]&=&-f_1(a)e_0-f_0(a)e_1,\\
		\left[a,b\right]&=&[a,b]_\ai+\varphi_0(a,b)e_0+\varphi_1(a,b)e_1,
	\end{array}
	\right.\\
\end{equation}
where $\varphi_k(a,b):=\w_\ai(\D_k(a),b)+\w_\ai(a,\D_k(b))$, for $k\in\{0,1\}$.\\

The next proposition outlines all the implications of $[\,\cdot\,,\cdot\,]$ being a Lie bracket on $\g$.

\begin{proposition}\label{propCondDoublExtDim2}
	The following assertions hold:\vspace{1mm}
	\item[$1.$] $(\ai,[\,\cdot\,,\cdot\,]_\ai,\w_\ai)$ is a symplectic Lie algebra$\;$\vspace{1mm}
	\item[$2.$] $\delta f_0=\delta f_1=0$, i.e., $f_0,f_1\in\mathcal{Z}^1(\ai,\rel)\,;$\vspace{1mm} 
	\item[$3.$] $\delta\varphi_0+\varphi_0\wedge f_0=-\varphi_1\wedge f_1\,;$\vspace{1mm}
	\item[$4.$] $\delta\varphi_1+\varphi_1\wedge f_0=\varphi_0\wedge f_1\,;$\vspace{1mm}
	\item[$5.$] $f_0(z_0)=\varepsilon_0\varepsilon_1+\varepsilon_2\varepsilon_3$,\, and $\,f_1(z_0)=\varepsilon_1(\varepsilon_2-\varepsilon_1)-\varepsilon_3(\varepsilon_0+\varepsilon_3)\,;$\vspace{1mm}
	\item[$6.$] $f_0\ro\D_0=-\varepsilon_0f_0-\varepsilon_2f_1$,\, and $\,f_1\ro\D_0=\varepsilon_1f_0+\varepsilon_3f_1\,;$\vspace{1mm}
	\item[$7.$] $f_0\ro\D_1=-\varepsilon_2f_0+\varepsilon_0f_1$,\, and $\,f_1\ro\D_1=\varepsilon_3f_0-\varepsilon_1f_1\,;$\vspace{1mm}
	\item[$8.$] For any $a,b\in\ai$, we have\footnote{Recall that, $\widetilde{\delta}_{_f}$ is the Chevalley-Eilenberg differential corresponding to the representation $\widetilde{\rho}_{_{f_0}}:\ai\to\gl(\ai)$ defined by $\widetilde{\rho}_{_{f_0}}(a):=\ad^\ai_a+f_0(a)\id_\ai$.}:
	\begin{equation*}
\left(\widetilde{\delta}_{_{f_0}}\D_0\right)(a,b)=f_1(a)\D_1(b)-f_1(b)\D_1(a),\et
\left(\widetilde{\delta}_{_{f_0}}\D_1\right)(a,b)=-f_1(a)\D_0(b)+f_1(b)\D_0(a)\,;
	\end{equation*}
	\item[$9.$] For any $a\in\ai$, we have:
	\begin{equation*}
	\begin{array}{r c l}\vspace{0.75mm}
	\ell_1\big(\D_0(a)\big)+\w_\ai\big(z^\ai,\D_0(a)\big)-\ell_0\big(\D_1(a)\big)+\varphi_0(z^\ai,a)&=&(\varepsilon_1-2\varepsilon_2)\ell_0(a)+(2\varepsilon_0+\varepsilon_3)\Big(\ell_1(a)+\w_\ai(z^\ai,a)\Big)-\varepsilon_1\ell_2(a)\\\vspace{2.5mm}
	&&+\,\varepsilon_3\ell_1(a)+3\varepsilon_4f_0(a)+\varepsilon_5f_1(a),\\\vspace{0.75mm}
	\ell_2\big(\D_0(a)\big)-\ell_1\big(\D_1(a)\big)+\varphi_1(z^\ai,a)&=&(\varepsilon_1-2\varepsilon_2)\ell_1(a)+(2\varepsilon_0+\varepsilon_3)\ell_2(a)+\varepsilon_1\Big(\ell_1(a)+\w_\ai(z^\ai,a)\Big)\\\vspace{2.5mm}
	&&-\,\varepsilon_3\ell_0(a)-\varepsilon_4f_1(a)+3\varepsilon_5f_0(a),\\
	\big[\D_0,\D_1\big]_{\gl(\ai)}(a)+2f_0(a)z^\ai&=&[z^\ai,a]_\ai+(\varepsilon_2-\varepsilon_1)\D_0(a)-(\varepsilon_0+\varepsilon_3)\D_1(a)\,;
	\end{array}
\end{equation*}
	\item[$10.$] For any $a,b\in\ai$, we have:
	\begin{equation*}
		\begin{array}{r c l}\vspace{0.75mm}
		\varphi_0(\D_0(a),b)+\varphi_0(a,\D_0(b))-\varepsilon_0\varphi_0(a,b)+\varepsilon_1\varphi_1(a,b)&=&\ell_0([a,b]_\ai)+2\left(\ell_0\wedge f_0\right)(a,b)+2\left(\ell_1\wedge f_1\right)(a,b)\\\vspace{2.5mm}
		&&+\,\big((i_{z^\ai}\w_\ai)\wedge f_1\big)(a,b),\\\vspace{0.75mm}
		\varphi_1(\D_0(a),b)+\varphi_1(a,\D_0(b))-\varepsilon_0\varphi_1(a,b)-\varepsilon_1\varphi_0(a,b)&=&\ell_1([a,b]_\ai)+2\left(\ell_1\wedge f_0\right)(a,b)+\left(\ell_2\wedge f_1\right)(a,b)\\\vspace{2.5mm}
		&&+\,(f_1\wedge \ell_0)(a,b),\\\vspace{0.75mm}
		\varphi_0(\D_1(a),b)+\varphi_0(a,\D_1(b))-\varepsilon_2\varphi_0(a,b)+\varepsilon_3\varphi_1(a,b)&=&\ell_1([a,b]_\ai)+\w_\ai(z^\ai,[a,b]_\ai)+\left(f_1\wedge \ell_0\right)(a,b)\\\vspace{2.5mm}
		&&+\left(\ell_2\wedge f_1\right)(a,b)+2\left(\ell_1\wedge f_0\right)(a,b)+2\big((i_{z^\ai}\w_\ai)\wedge f_0\big)(a,b),\\\vspace{0.75mm}
		\varphi_1(\D_1(a),b)+\varphi_1(a,\D_1(b))-\varepsilon_2\varphi_1(a,b)-\varepsilon_3\varphi_0(a,b)&=&\ell_2([a,b]_\ai)+2\left(f_1\wedge \ell_1\right)(a,b)+2\left(\ell_2\wedge f_0\right)(a,b)\\
		&&+\left(\ell_2\wedge f_1\right)(a,b)+\big(f_1\wedge(i_{z^\ai}\w_\ai)\big)(a,b).\end{array}
	\end{equation*}
\end{proposition}

\begin{proof}
	Let $a,b,c\in\ai$, we compute:
	\begin{eqnarray*}
		\sum_{\circlearrowright (a,b,c)} \big[a,[b,c]\big]&=&\sum_{\circlearrowright (a,b,c)} \Big[a,[b,c]_\ai+\varphi_0(b,c)e_0+\varphi_1(b,c)e_1\Big]\\
		&=&\sum_{\circlearrowright (a,b,c)} \big[a,[b,c]_\ai\big]+\sum_{\circlearrowright (a,b,c)} \varphi_0(b,c)[a,e_0]+\sum_{\circlearrowright (a,b,c)} \varphi_1(b,c)[a,e_1]\\
		&=&\sum_{\circlearrowright (a,b,c)} \big[a,[b,c]_\ai\big]_\ai+\sum_{\circlearrowright (a,b,c)} \varphi_0(a,[b,c]_\ai)e_0+\sum_{\circlearrowright (a,b,c)} \varphi_1(a,[b,c]_\ai)e_1+\sum_{\circlearrowright (a,b,c)}\varphi_0(b,c)f_0(a)e_0\\
		&&-\sum_{\circlearrowright (a,b,c)}\varphi_0(b,c)f_1(a)e_1+\sum_{\circlearrowright (a,b,c)}\varphi_1(b,c)f_1(a)e_0+\sum_{\circlearrowright (a,b,c)}\varphi_1(b,c)f_0(a)e_1\\
		&=&\sum_{\circlearrowright (a,b,c)} \big[a,[b,c]_\ai\big]_\ai+\big(\delta\varphi_0+\varphi_0\wedge f_0+\varphi_1\wedge f_1\big)(a,b,c)e_0+\big(\delta\varphi_1+\varphi_1\wedge f_0-\varphi_0\wedge f_1\big)(a,b,c)e_1.
	\end{eqnarray*}
	Similarly, we have:
	\begin{equation*}
\begin{array}{r c l}\vspace{0.75mm}
\big[e_0,[a,b]\big]+\big[b,[e_0,a]\big]+\big[a,[b,e_0]\big]&=&\Big[e_0,[a,b]_\ai+\varphi_0(a,b)e_0+\varphi_1(a,b)e_1\Big]\\\vspace{1.5mm}
		&&+\,\Big[b,-f_0(a)e_0+f_1(a)e_1\Big]+\Big[a,f_0(b)e_0-f_1(b)e_1\Big]\\\vspace{0.75mm}
		&=&-f_0([a,b]_\ai)e_0+f_1([a,b]_\ai)e_1-f_0(a)f_0(b)e_0+f_0(a)f_1(b)e_1\\\vspace{0.75mm}
		&&+\,f_1(a)f_1(b)e_0+f_1(a)f_0(b)e_1+f_0(b)f_0(a)e_0-f_0(b)f_1(a)e_1\\\vspace{1.5mm}
		&&-\,f_1(b)f_1(a)e_0-f_1(b)f_0(a)e_1\\
		&=&-\,f_0([a,b]_\ai)e_0+f_1([a,b]_\ai)e_1.
	\end{array}
\end{equation*}
	Analogously, we compute:
	\begin{equation*}
		\begin{array}{r c l}\vspace{0.75mm}
\big[[d_0,d_1],e_0\big]+\big[[e_0,d_0],d_1\big]+\big[[d_1,e_0],d_0\big]&=&(\varepsilon_2-\varepsilon_1)[d_0,e_0]-(\varepsilon_0+\varepsilon_3)[d_1,e_0]+[a_0,e_0]\\\vspace{1.5mm}
		&&-\,\varepsilon_0[e_0,d_1]-\varepsilon_1[e_1,d_1]+\varepsilon_2[e_0,d_0]+\varepsilon_3[e_1,d_0]\\
		&=&\Big(f_0(z_0)-\varepsilon_0\varepsilon_1-\varepsilon_2\varepsilon_3\Big)e_0+\Big(-f_1(z_0)-\varepsilon_1^2-\varepsilon_3^2+\varepsilon_1\varepsilon_2-\varepsilon_0\varepsilon_3\Big)e_1.
	\end{array}
\end{equation*}
	Moreover,
	\begin{equation*}
		\begin{array}{r c l}\vspace{0.75mm}
\big[[d_0,e_0],a\big]+\big[e_0,[d_0,a]\big]-\big[d_0,[e_0,a]\big]&=&\varepsilon_0[e_0,a]+\varepsilon_1[e_1,a]+f_0(a)[e_0,d_0]+f_1(a)[e_0,d_1]\\\vspace{1.5mm}
		&&+\,[e_0,\D_0(a)]+f_0(a)[d_0,e_0]-f_1(a)[d_0,e_1]\\
		&=&\Big(-f_0(\D_0(a))-\varepsilon_2f_1(a)-\varepsilon_0f_0(a)\Big)e_0+\Big(f_1(\D_0(a))-\varepsilon_3f_1(a)-\varepsilon_1f_0(a)\Big)e_1.
	\end{array}
\end{equation*}
	In addition,
	\begin{equation*}
		\begin{array}{r c l}\vspace{0.75mm}
\big[[d_1,e_0],a\big]+\big[e_0,[d_1,a]\big]-\big[d_1,[e_0,a]\big]&=&\varepsilon_2[e_0,a]+\varepsilon_3[e_1,a]-f_1(a)[e_0,d_0]+f_0(a)[e_0,d_1]\\\vspace{1.5mm}
	&&+\,[e_0,\D_1(a)]+f_0(a)[d_1,e_0]-f_1(a)[d_1,e_1]\\
		&=&\Big(-f_1(\D_1(a))-\varepsilon_1f_1(a)+\varepsilon_3f_0(a)\Big)e_0-\Big(f_0(\D_1(a))+\varepsilon_2f_0(a)-\varepsilon_0f_1(a)\Big)e_1.
	\end{array}
\end{equation*}
	Furthermore,
	\begin{equation*}
		\begin{array}{r c l}\vspace{0.75mm}
\big[a,[d_0,d_1]\big]-\big[d_1,[d_0,a]\big]+\big[d_0,[d_1,a]\big]&=&\Big[a,(\varepsilon_2-\varepsilon_1)d_0-(\varepsilon_0+\varepsilon_3)d_1+z^\ai+\varepsilon_4e_0+\varepsilon_5e_1\Big]\\\vspace{0.75mm}
&&-\,\Big[d_1,f_0(a)d_0+f_1(a)d_1+\D_0(a)+\ell_0(a)e_0+\ell_1(a)e_1\Big]\\\vspace{1.5mm}
&&+\,\Big[d_0,-f_1(a)d_0+f_0(a)d_1+\D_1(a)+\ell_3(a)e_0+\ell_2(a)e_1\Big]\\\vspace{0.75mm}
&=&\Big( (\varepsilon_2-\varepsilon_1)f_0(a)-(\varepsilon_0+\varepsilon_3)f_1(a)+f_1(\D_0(a))+f_0(\D_1(a))\Big)d_0\\\vspace{0.75mm}
&&+\,\Big( (\varepsilon_1-\varepsilon_2)f_1(a)-(\varepsilon_0+\varepsilon_3)f_0(a)+f_1(\D_1(a))-f_0(\D_0(a))\Big)d_1\\\vspace{0.75mm}
&&+\,(\varepsilon_1-\varepsilon_2)\D_0(a)+(\varepsilon_0+\varepsilon_3)\D_1(a)+[a,z^\ai]_\ai+2f_0(a)z^\ai+\big[\D_0,\D_1\big]_{\gl(\ai)}(a)\\\vspace{0.75mm}
&&+\,\Big((\varepsilon_1-2\varepsilon_2)\ell_0(a)+(2\varepsilon_0+\varepsilon_3)\ell_3(a)-\varepsilon_1\ell_2(a)+\varepsilon_3\ell_1(a)+3\varepsilon_4f_0(a)\\\vspace{0.75mm}
&&+\,\varepsilon_5f_1(a)-\ell_3(\D_0(a))+\ell_0(\D_1(a))+\varphi_0(a,z^\ai)\Big)e_0\\\vspace{0.75mm}
&&+\,\Big((\varepsilon_1-2\varepsilon_2)\ell_1(a)+(2\varepsilon_0+\varepsilon_3)\ell_2(a)+\varepsilon_1\ell_3(a)-\varepsilon_3\ell_0(a)-\varepsilon_4f_1(a)\\
&&+\,3\varepsilon_5f_0(a)-\ell_2(\D_0(a))+\ell_1(\D_1(a))+\varphi_1(a,z^\ai)\Big)e_1,
	\end{array}
	\end{equation*}
with $\ell_3:=\ell_1+i_{z^\ai}\w_\ai$. On the other hand,
	\begin{equation*}
		\begin{array}{r c l}\vspace{0.75mm}
\big[d_0,[a,b]\big]+\big[b,[d_0,a]\big]+\big[[d_0,b],a\big]&=&\big[d_0,[a,b]_\ai\big]+\varphi_0(a,b)[d_0,e_0]+\varphi_1(a,b)[d_0,e_1]+f_0(a)[b,d_0]\\\vspace{0.75mm}
&&+\,f_1(a)[b,d_1]+[b,\D_0(a)]+\ell_0(a)[b,e_0]+\ell_1(a)[b,e_1]+f_0(b)[d_0,a]\\\vspace{1.5mm}
		&&+\,f_1(b)[d_1,a]+[\D_0(b),a]+\ell_0(b)[e_0,a]+\ell_1(b)[e_1,a]\\\vspace{0.75mm}
		&=&\D_0([a,b]_\ai)-f_0(a)\D_0(b)-f_1(a)\D_1(b)+[b,\D_0(a)]_\ai+f_0(b)\D_0(a)\\\vspace{0.75mm}
		&&+\,f_1(b)\D_1(a)+[\D_0(b),a]_\ai+\Big(\ell_0([a,b]_\ai)+\varepsilon_0\varphi_0(a,b)-\varepsilon_1\varphi_1(a,b)\\\vspace{0.75mm}
		&&-\,2f_0(a)\ell_0(b)+2f_0(b)\ell_0(a)+\ell_1(a)f_1(b)-\ell_1(b)f_1(a)-f_1(a)\ell_3(b)\\\vspace{0.75mm}
		&&+\,f_1(b)\ell_3(a)+\varphi_0(b,\D_0(a))+\varphi_0(\D_0(b),a)\Big)e_0+\Big(\ell_1([a,b]_\ai)\\\vspace{0.75mm}
		&&+\,\varepsilon_1\varphi_0(a,b)+\varepsilon_0\varphi_1(a,b)-2f_0(a)\ell_1(b)+2f_0(b)\ell_1(a)-f_1(a)\ell_2(b)\\
		&&+\,f_1(b)\ell_2(a)-\ell_0(a)f_1(b)+\ell_0(b)f_1(a)+\varphi_1(b,\D_0(a))+\varphi_1(\D_0(b),a)\Big)e_1.
	\end{array}
	\end{equation*}
	Finally, 
	\begin{equation*}
		\begin{array}{r c l}\vspace{0.75mm}
\big[d_1,[a,b]\big]-\big[[d_1,a],b\big]+\big[[d_1,b],a\big]&=&\big[d_1,[a,b]_\ai\big]+\varphi_0(a,b)[d_1,e_0]+\varphi_1(a,b)[d_1,e_1]+f_1(a)[d_0,b]\\\vspace{0.75mm}
&&-\,f_0(a)[d_1,b]-[\D_1(a),b]-\ell_3(a)[e_0,b]-\ell_2(a)[e_1,b]-f_1(b)[d_0,a]\\\vspace{1.5mm}
		&&+\,f_0(b)[d_1,a]+[\D_1(b),a]+\ell_3(b)[e_0,a]+\ell_2(b)[e_1,a]\\\vspace{0.75mm}
		&=&\D_1([a,b]_\ai)+f_1(a)\D_0(b)-f_0(a)\D_1(b)-[\D_1(a),b]_\ai-f_1(b)\D_0(a)\\\vspace{0.75mm}
		&&+\,f_0(b)\D_1(a)+[\D_1(b),a]_\ai+\Big(\ell_3([a,b]_\ai)+\varepsilon_2\varphi_0(a,b)-\varepsilon_3\varphi_1(a,b)\\\vspace{0.75mm}
		&&-\,2f_0(a)\ell_3(b)+2f_0(b)\ell_3(a)+f_1(a)\ell_0(b)-f_1(b)\ell_0(a)+\ell_2(a)f_1(b)\\\vspace{0.75mm}
		&&-\,\ell_2(b)f_1(a)-\varphi_0(\D_1(a),b)+\varphi_0(\D_1(b),a)\Big)e_0+\Big(\ell_2([a,b]_\ai)\\\vspace{0.75mm}
		&&+\,\varepsilon_3\psi_0(a,b)+\varepsilon_2\psi_1(a,b)-2f_0(a)\ell_2(b)+2f_0(b)\ell_2(a)+f_1(a)\ell_1(b)\\
		&&-\,f_1(b)\ell_1(a)+f_1(a)\ell_3(b)-f_1(b)\ell_3(a)-\varphi_1(\D_1(a),b)+\varphi_1(\D_1(b),a)\Big)e_1,
	\end{array}
\end{equation*}
and therefore, the equivalence follows.
\end{proof}

\begin{definition}
	The symplectic Lie algebra $(\ai,[\,\cdot\,,\cdot\,]_\ai,\w_\ai)$ given in Proposition \ref{propCondDoublExtDim2} is called the \textit{symplectic reduction} of $(\g,[\,\cdot\,,\cdot\,],\w)$ with respect to the isotropic ideal $\mathfrak{i}$.
\end{definition}

\begin{remark}
Since $(\g,\cro)$ is a solvable Lie algebra, and $(\ai,\cro_\ai)$ is isomorphic (as Lie algebra) to $\ii^\perp/\ii$, one can easily deduce that $(\ai,\cro_\ai)$ is solvable. Moreover, a straightforward computation, using \eqref{eqSolDim2II}, gives:
	\begin{equation*}
	\kappa_\g(a,b)=\kappa_\ai(a,b)+4f_0(a)f_0(b)-4f_1(a)f_1(b),
\end{equation*} 
for any $a,b\in\ai$. 
\end{remark}

Now, we are in position to state our main result of this section. Since $\w_\ai$ is nondegenerate and $\delta f_0=\delta f_1=0$, there exist a unique $(a_0^\ai,a_1^\ai,b_0^\ai,b_1^\ai,b_2^\ai)\in[\ai,\ai]^\perp\times[\ai,\ai]^\perp\times\ai\times\ai\times\ai$ such that:
\begin{equation*}
f_k=i_{a^\ai_k}\w_\ai,\et \ell_{k'}=i_{b^\ai_{k'}}\w_\ai,
\end{equation*}
with $k\in\{0,1\}$ and $k'\in\{0,1,2\}$. Moreover, using Proposition \ref{propCondDoublExtDim2}, we can easily check that $(a^\ai_0,a^\ai_1,b^\ai_0,b^\ai_1,b^\ai_2,z^\ai,\D_0,\D_1,\varepsilon_0,\ldots,\varepsilon_5)$ is admissible (in the sense of Definition \ref{AdmiSDIm2}). So, combining this with Theorem \ref{TheoDoubExtdim1} and Corollary \ref{CorSec2-1UnimSolv}, we have proved the following:

\begin{theorem}\label{TheoDoubExtdim2}
Let $(\g,\cro,\w)$ be a solvable symplectic Lie algebra possessing an isotropic ideal. Then, $(\g,\cro,\w)$ is a symplectic double extension by a line or a plane of a solvable symplectic Lie algebra $(\ai,\cro_\ai,\w_\ai)$.
\end{theorem}

\begin{corollary}\label{CorDoubExtdim2}
Any solvable symplectic Lie algebra possessing an isotropic ideal can be obtained through a finite sequence of symplectic double extensions by a line or a plane, starting from a solvable symplectic Lie algebra without any nonzero isotropic ideal\footnote{Such symplectic Lie algebra is called \textit{symplectically irreducible} (see Definition \ref{DefSymplecIrreduciSec6}).}.
\end{corollary}

Since any unimodular symplectic Lie algebra is solvable (see Theorem \ref{Chu'sTh} bellow), then combining Theorem \ref{TheoDoubExtdim2} with Corollary \ref{CorSec2-1UnimSolv} and Proposition \ref{PropSymLiAlgDoblExtDim2Prop01}, we get:

\begin{corollary}\label{TheoDoubExtdmi1dim2Unimodlr}
Let $(\g,\cro,\w)$ be a unimodular symplectic Lie algebra which has an isotropic ideal. Then, $(\g,\cro,\w)$ is a symplectic double extension by a line or a plane of a unimodular symplectic Lie algebra $(\ai,\cro_\ai,\w_\ai)$.
\end{corollary}

\begin{corollary}
Any unimodular symplectic Lie algebra possessing an isotropic ideal can be obtained through a finite sequence of symplectic double extensions by a line or a plane, starting from a unimodular symplectic Lie algebra without any nonzero isotropic ideal.
\end{corollary}

We conclude this section by presenting a sufficient condition for a symplectic Lie algebra to be unimodular and, in particular, solvable. Let us begin with the following lemma, which provides a clear description of the symplectic complement of the derived ideal in a symplectic Lie algebra.

\begin{lemma}\label{LemCentr}
	For any symplectic Lie algebra $(\g,[\,\cdot\,,\cdot\,],\w)$, the symplectic complement of $[\g,\g]$ is given by: 
	\begin{equation}\label{ComIdealOrthog}
		[\g,\g]^\perp=\Big\{x\in\g\,|\,\,\ad_x^\aj=-\ad_x\Big\}.
	\end{equation} 
	In particular, $[\g,\g]^\perp$contains the center $\z(\g)$ of $\g$. Moreover, $[\g,\g]^\perp$ is an abelian Lie subalgebra of $(\g,\cro)$.
\end{lemma}

\begin{proof}
	The equality \eqref{ComIdealOrthog} follows directly from the fact that $\w\in\mathcal{Z}^2(\g,\rel)$. For the last assertion, let $x,y\in[\g,\g]^\perp$, and $z\in\g$, we have:
	\begin{eqnarray*}
		\w([x,y],z)=\w([x,z],y)+\w(x,[y,z])=0.
	\end{eqnarray*}
	Thus, $[x,y]=0$, and hence $[\g,\g]^\perp$ is an abelian Lie subalgebra of $(\g,\cro)$.
\end{proof}

The next lemma shows that for a non-unimodular symplectic Lie algebra $(\g,[\,\cdot\,,\cdot\,],\w)$ the derived ideal $[\g,\g]$ is always degenerate.

\begin{lemma}\label{lema[g,g]degenerate}
	For any symplectic Lie algebra $(\g,[\,\cdot\,,\cdot\,],\w)$, we have:
	\begin{equation*}
		u^\g\in[\g,\g]\cap [\g,\g]^\perp.
	\end{equation*}
\end{lemma}

\begin{proof}
	It is clear that $u^\g\in[\g,\g]^\perp$, since for any $x,y\in\g$, we have:
	\begin{equation*}
		\w(u^\g,[x,y])=\tr(\ad_{[x,y]})=\tr\left([\ad_x,\ad_y]\right)=0.
	\end{equation*}
	On the other hand, $u^\g\in [\g,\g]$, since for $x\in[\g,\g]^\perp$, we have:
	\begin{equation*}
		\omega\left(u^\mathfrak{g},x\right)=\tr\left(\ad_x\right)=0,
	\end{equation*}
	where the last equality follows from \eqref{ComIdealOrthog} and the fact that $\tr(F^\divideontimes)=\tr(F)$ for any linear endomorphism $F\in\gl(\g)$. Hence, $u^\g\in[\g,\g]\cap [\g,\g]^\perp$. 
\end{proof}

The following theorem is due to Chu (cf. \cite[p. 157]{Chu}). Although the proof closely follows Chu's, we work here with linear representations instead of affine representations.

\begin{theorem}[\bf Chu]\label{Chu'sTh}
	A $2n$-dimensional unimodular symplectic Lie algebra $(\g,[\,\cdot\,,\cdot\,],\w)$ is solvable.
\end{theorem}

\begin{proof}
	Suppose by contradiction that $(\g,[\,\cdot\,,\cdot\,])$ is non-solvable. Then, $\g$ has a Levi decomposition $\g=\mathfrak{s}\oplus\mathfrak{rad}(\g)$, with $\mathfrak{rad}(\g)$ denotes its radical and $\s$ a nonzero semisimple Lie subalgebra of it. Define a representation of $\s$ in $\g^*$ as follows:
	\begin{equation*}
		\rho:\s\to\gl(\g^*),\quad a\mapsto\rho(a),
	\end{equation*}
	with $\rho(a):\g^*\to\g^*$ given by $\rho(a)(\psi):=i_a(\delta\psi).$
	If we denote by $\widetilde{\w}:\s\to\g^*$ the linear map defined by $\widetilde{\w}(a):=i_a\w$, then one can easily verify that $\delta_\rho\widetilde{\w}=0$. Thus, by Whitehead Lemma (cf. \cite[p. $220$]{Var}), there exists $\theta\in\g^*$ such that: $\widetilde{\w}(a)=\rho(a)(\theta)$, for any $a\in\s$.
	In other words,
	\begin{equation*}
		i_a\left(\w-\delta\theta\right)=0,\qquad\forall\,a\in\s.
	\end{equation*}
	In particular, $i_a\left(\w-\delta\theta\right)^n=0$ for any $a\in\s$. This implies that $\left(\w-\delta\theta\right)^n=0$, and therefore\footnote{$\begin{psmallmatrix}
			n\\
			k
		\end{psmallmatrix}:=\tfrac{n!}{(n-k)!k!}$.}
	\begin{eqnarray*}
		\w^n&=&\somkn(-1)^{k+1}\begin{psmallmatrix}
			n\\
			k
		\end{psmallmatrix}\,\w^{n-k}\wedge\left(\delta\theta\right)^k\\
		&=&\delta\left(\somkn(-1)^{k+1}\begin{psmallmatrix}
			n\\
			k
		\end{psmallmatrix}\,\w^{n-k}\wedge\theta\wedge\left(\delta\theta\right)^{k-1}\right).
	\end{eqnarray*}
	Consequently, $\w^n$ is a $2n$-coboundary of $\g$. However, since $(\g,\cro)$ is unimodular, it follows that $H^{2n}(\g,\rel)\neq\{0\}$ (cf. \cite[p. 206]{HilNeb}), leading to a contradiction. This establishes that $(\g,[\,\cdot\,,\cdot\,])$ must be solvable.
\end{proof}

The next theorem shows that symplectic Lie algebras with nondegenerate derived ideal are solvable.

\begin{theorem}\label{MainResNot}
	Let $(\g,[\,\cdot\,,\cdot\,],\w)$ be a symplectic Lie algbera such that $[\g,\g]$ is nondegenerate. Then, $(\g,[\,\cdot\,,\cdot\,])$ is unimodular, and therefore it is solvable. Moreover, if $(\g,[\,\cdot\,,\cdot\,])$ is non-abelian, then it is non-nilpotent.
\end{theorem}

\begin{proof}
	If $(\g,[\,\cdot\,,\cdot\,],\w)$ is a symplectic Lie algbera such that $[\g,\g]$ is nondegenerate, then by  Lemma \ref{lema[g,g]degenerate}, $(\g,[\,\cdot\,,\cdot\,],\w)$ is a unimodular symplectic Lie algebra. Hence, the result follows by applying Theorem \ref{Chu'sTh}. To prove that $(\g,\cro)$ cannot be nilpotent, we will proceed by contradiction. Assume that $(\g,\cro)$ is $k+1$-step nilpotent with $k\in\mathbb{N}\backslash\{0\}$. Then, $\mathcal{C}^k(\g)\neq \{0\}$, $\mathcal{C}^k(\g)\subseteq\z(\g)$, and $\mathcal{C}^k(\g)\subseteq [\g,\g]$, where $\{\mathcal{C}^i(\g)\}_{i\in\mathbb{N}}$ denotes the lower central series of $(\g,\cro)$. Thus, using the fact that $\z(\g)\subseteq [\g,\g]^\perp$ (see Lemma \ref{LemCentr}), we deduce that $[\g,\g]^\perp\cap [\g,\g]\neq\{0\}$. However, this contradicts the fact that $[\g,\g]$ is nondegenerate, and hence completing the proof.
\end{proof}

\section{Symplectically irreducible symplectic Lie algebras}\label{Art4SectionIrreducible}

In this section, we provide a complete characterization of symplectically irreducible symplectic Lie algebras. In particular, we present a new algebraic proof of a structural theorem for symplectically irreducible symplectic Lie algebras. Additionally, we classify all such Lie algebras of dimension up to $6$ that admit such a structure.

To begin, according to Theorem \ref{TheoDoubExtdim2}, any solvable symplectic Lie algebra $(\g,\cro,\w)$ with an isotropic ideal is a symplectic double extension by either a line or a plane of a solvable symplectic Lie algebra $(\ai,\cro_\ai,\w_\ai)$. Since the symplectic reduction $(\ai,\cro_\ai,\w_\ai)$ of $(\g,\cro,\w)$ remains solvable, one may consider applying the same process iteratively. However, the feasibility of repeating this procedure depends solely on the existence of an isotropic ideal (cf. Corollary \ref{CorDoubExtdim2}). This motivates the following definition:

\begin{definition}\label{DefSymplecIrreduciSec6}
	A symplectic Lie algebra $(\g,[\,\cdot\,,\cdot\,],\w)$ is called \textit{symplectically irreducible} if $(\g,[\,\cdot\,,\cdot\,],\w)$ does not contain any nonzero isotropic ideal.
\end{definition}

As an immediate consequence, we have the following proposition:

\begin{proposition}
	A symplectically irreducible symplectic Lie algebra $(\g,\cro,\w)$ has a zero center, i.e., $\mathfrak{z}(\g)=\{0\}$. In particular, a nonzero symplectically irreducible symplectic Lie algebra cannot be abelian.
\end{proposition}

The following theorem, which was established in \cite[p. 38]{Bau}, provides us with a sufficient condition for a symplectic Lie algebra to be symplectically irreducible.

\begin{theorem}\label{T1SymplIrredu00}
Let $(\g,\cro,\w)$ be a symplectic Lie algebra with trivial center, i.e., $\z(\g)=\{0\}$. If $[\g,\g]$ admits an orthogonal direct sum decomposition into pairwise non-isomorphic $2$-dimensional symplectic ideals of $(\g,\cro,\w)$, each of which contains no $1$-dimensional ideal of $(\g,\cro)$, then $(\g,\cro,\w)$ is symplectically irreducible.
\end{theorem}

\begin{proof}
Let $m\in\nat\backslash\{0\}$ be a positive integer such that:
\begin{equation*}
[\g,\g]=\ii_1\oplus\cdots\oplus\ii_m,
\end{equation*}
where each $\ii_k$ is a symplectic $2$-dimensional ideal of $(\g,\cro,\w)$ which contains no $1$-dimensional ideal of $(\g,\cro)$. It is obvious that $[\g,\g]$ is nondegenerate, and hence $\g=[\g,\g]^\perp\oplus [\g,\g]$. Let $\jj$ be an isotropic ideal of $(\g,\cro,\w)$, then using the fact that $\ii_k$ is symplectic and contains no $1$-dimensional ideal, we get that $\jj\cap\ii_k=\{0\}$ for any $k\in\{1,\ldots,m\}$. Moreover, since $\ii_k$'s are pairwise non-isomorphic\footnote{This means that for $k'\neq k$, there is no linear isomorphism between $\ii_k$ and $\ii_{k'}$ which commutes with $\ad_x\in\gl(\g)$, for all $x\in\g$.}, it follows that $\jj\cap [\g,\g]=\{0\}$. As a result,
\begin{equation*}
[\jj,\g]\subseteq \big[\jj,[\g,\g]^\perp\big]+\big[\jj,[\g,\g]\big]\subseteq\jj\cap [\g,\g]=\{0\}.
\end{equation*}
Therefore, $\jj\subseteq\z(\g)=\{0\}$, and this shows that $(\g,\cro,\w)$ has no nonzero isotropic ideal.
\end{proof}

\begin{example}\label{ExSympIrrDim6}
	Let $(\g_6,[\,\cdot\,,\cdot\,],\w)$ be the $6$-dimensional symplectic Lie algebra defined by
	\begin{equation*}
		[e_1,e_3]= e_4,\quad [e_1,e_4]=-e_3,\quad [e_2,e_5]= e_6,\et [e_2,e_6]=-e_5,
	\end{equation*} 
where the other Lie brackets are either zero or given by skew-symmetry, and the symplectic form $\w$ on it is given by: $$\w:=e_1^*\wedge e_2^*+e_3^*\wedge e_4^*+e_5^*\wedge e_6^*.$$ It is clear that $\z(\g_6)=\{0\}$, and $\,[\g_6,\g_6]=\Span_\rel\{e_3,e_4,e_5,e_6\}$ is a nondegenerate abelian ideal of $(\g_6,\cro,\w)$. Moreover,  $[\g_6,\g_6]=\ii_1\oplus\ii_2$ with $\ii_k:=\Span_\rel\{e_{2k+1},e_{2(k+1)}\}$, is an orthogonal direct sum of $2$-dimensional ideals of $(\g_6,\cro,\w)$ and each ideal of this decomposition is nondegenerate and contains no $1$-dimensional ideal of $(\g_6,\cro)$. Further, one can easily check that $\ii_1$ and $\ii_2$ are not isomorphic. Hence, $(\g_6,\cro,\w)$ is symplectically irreducible.
\end{example}

Now, we will state and prove the main theorem of this section. In fact, we will give a new (completely algebraic) proof of \cite[Proposition 2.4.2, p. 37]{Bau}. This Proposition was given without proof; however the authors have mentioned that it is a direct consequence of \cite[Theorem 1.3, p. 211]{DM}.

\begin{theorem}\label{T1SymplIrredu}
	Let $(\g,[\,\cdot\,,\cdot\,],\w)$ be a nonzero symplectically irreducible symplectic Lie algebra. Then, the following hold:
	\begin{itemize}
		\item[$1.$] $(\g,\cro)$ is a $2$-step solvable Lie algebra, i.e., $\left[[\g,\g],[\g,\g]\right]=\{0\}$;
		\item[$2.$] $\g=[\g,\g]\oplus[\g,\g]^\perp$, where $[\g,\g]$ is a nondegenerate abelian ideal of $(\g,\cro,\w)$, and $[\g,\g]^\perp$ is a nondegenerate abelian Lie subalgebra of $(\g,\cro,\w)$. In particular, $(\g,\cro)$ is unimodular;
		\item[$3.$] $[\g,\g]$ is a maximal abelian ideal of $(\g,\cro)$;
		\item[$4.$] If we denote by $\pi$ the representation of $[\g,\g]^\perp$ on $[\g,\g]$ defined by $\pi(a)(x):= [a,x],\,$  then $\pi$ is a faithful representation such that $\w\left(\pi(a)(x),y\right)= -\w\left(x,\pi(a)(y)\right)$ for all $a \in [\g,\g]^\perp$ and $x, y \in [\g,\g]\,$ $($we say that the action of $[\g,\g]^\perp$ on $[\g,\g]$ is faithful and skew-symmetric$)$.
	\end{itemize}
\end{theorem}

\begin{proof}
	For the first assertion, assume by contradiction that $(\g,\cro)$ is not $2$-step solvable, which means that $\left[[\g,\g],[\g,\g]\right]\neq\{0\}$. Then, the radical $\mathfrak{rad}([\g,\g])$ of $[\g,\g]$ is nonzero. Indeed, since $[\g,\mathfrak{rad}(\g)]=\mathfrak{rad}([\g,\g])$, it follows that if $\mathfrak{rad}([\g,\g])=\{0\}$, then $(\g,\cro,\w)$ will be a reductive symplectic Lie algbera, and therefore $(\g,\cro)$ must be abelian. But this is impossible since $(\g,\cro,\w)$ is symplectically irreducible. Hence, $\mathfrak{rad}([\g,\g])\neq\{0\}$. Denote by $\mathfrak{z}_1:=\z\left(\mathfrak{rad}([\g,\g])\right)$ the center of the radical of $[\g,\g]$. It is easy to see that $\z_1$ is an ideal of $(\g,\cro)$, and further $\z_1\neq\{0\}$, since $\mathfrak{rad}([\g,\g])$ is nilpotent. Let $\mathfrak{t}:=\z_1\cap\z_1^\perp$, then by using the fact that $\z_1$ is abelian and $\w$ is a $2$-cocycle, one can check that $\mathfrak{t}$ is an isotropic ideal of $(\g,\cro,\w)$. Thus $\mathfrak{t}=\{0\}$, and therefore $\g=\z_1\oplus\ai$, where $\ai:=\z_1^\perp$. It is clear that $(\ai,\cro_{|\ai\times\ai},\w_{|\ai\times\ai})$ is a symplectic Lie subalgebra of $(\g,\cro,\w)$. Furthermore, for a Levi decomposition $\ai=\mathfrak{s}_\ai\oplus\mathfrak{rad}(\ai)$ of $\ai$, we have $[\ai,\ai]=\mathfrak{s}_\ai\oplus[\ai,\mathfrak{rad}(\ai)]$, and hence,
	\begin{eqnarray*}
		[\g,\g]=[\z_1,\ai]\oplus[\ai,\ai]=\z_1\oplus[\ai,\ai]=\mathfrak{s}_\ai\oplus\z_1\oplus[\ai,\mathfrak{rad}(\ai)].
	\end{eqnarray*}
	Thus, since $\z_1\oplus[\ai,\mathfrak{rad}(\ai)]$ is a solvable ideal of $[\g,\g]$, and $\mathfrak{s}_\ai$ is a semisimple Lie subalgebra of $(\g,\cro)$, we conclude that $\mathfrak{rad}([\g,\g])=\z_1\oplus[\ai,\mathfrak{rad}(\ai)]$. In addition, using the fact that $\z_1$ is the center of $\mathfrak{rad}([\g,\g])$, we obtain that $[\ai,\mathfrak{rad}(\ai)]$ is an ideal of $\mathfrak{rad}([\g,\g])$. Now, if $[\ai,\mathfrak{rad}(\ai)]=\{0\}$, then $(\ai,\cro_{|\ai\times\ai},\w_{|\ai\times\ai})$ will be a reductive symplectic Lie algebra, and hence $(\ai,\cro_{|\ai\times\ai})$ must be abelian. But in this case, $[\g,\g]$ will be equal to $\z_1$, which implies that $(\g,\cro)$ is $2$-step solvable. Thus a contradiction, and therefore $[\ai,\mathfrak{rad}(\ai)]\neq\{0\}$. Let $\z_2:=\z\left([\ai,\mathfrak{rad}(\ai)]\right)$ be the center of $[\ai,\mathfrak{rad}(\ai)]$. Since $[\ai,\mathfrak{rad}(\ai)]$ is nonzero and nilpotent, we have $\z_2\neq\{0\}$. On the other hand, using that $\mathfrak{rad}([\g,\g])=\z_1\oplus[\ai,\mathfrak{rad}(\ai)]$, and $\z_1$ is the center of $\mathfrak{rad}([\g,\g])$, we can easily deduce that $\z_2\subseteq\z_1$. Hence, $\z_2\subseteq\z_1\cap[\ai,\mathfrak{rad}(\ai)]$, which gives that $\z_1\cap[\ai,\mathfrak{rad}(\ai)]\neq\{0\}$. But this is a contradiction, since $\z_1\cap[\ai,\mathfrak{rad}(\ai)]\subseteq\z_1\cap\ai=\{0\}$. In summary, we have proved that $(\g,\cro)$ must be a $2$-step solvable Lie algebra. The second assertion follows easily from the fact that $[\g,\g]\cap[\g,\g]^\perp$ is an isotropic ideal of $(\g,\cro,\w)$. For the third assertion, let $\jj\subseteq\g$ be an abelian ideal such that $[\g,\g]\subseteq\jj$. By using the fact that $\jj$ and $[\g,\g]^\perp$ are both abelian, we can easily deduce that:
	\begin{equation*}
		\left[\jj\cap [\g,\g]^\perp,[\g,\g]\right]=\left[\jj\cap [\g,\g]^\perp,[\g,\g]^\perp\right]=\{0\}.
	\end{equation*}
	Therefore, $\jj\cap [\g,\g]^\perp\subseteq\z(\g)$, and hence $\jj\cap [\g,\g]^\perp=\{0\}$. As a result, $\jj$ must be equal to $[\g,\g]$. For the last assertion, if $\pi(a)=0$ for $a\in[\g,\g]^\perp$, then $[a,[\g,\g]]=\{0\}$. Further, since $[\g,\g]^\perp$ is abelian, we also have $\left[a,[\g,\g]^\perp\right]=\{0\}$. Thus, $a\in\z(\g)$, and therefore $a=0$. Finally, the skew-symmetry of the action of $[\g,\g]^\perp$ on $[\g,\g]$ can be easily checked by using the facts that $\w$ is a $2$-cocycle and $(\g,\cro)$ is $2$-step solvable.
\end{proof}

Note that Theorem \ref{T1SymplIrredu} was partially established by a different approach in \cite[pp. 777–778]{DM2}.\\
%\begin{remark}
%It is important to pointed out that the converse of Theorem \ref{T1SymplIrredu} holds and it was established in \cite[Theorem 2.4.3, p. 38]{Bau}.
%\end{remark}

The following corollary can be viewed as a converse to Theorem \ref{T1SymplIrredu00}.

\begin{corollary}\label{CompCharaOfIrrSymLiAlg0}
Let $(\g,[\,\cdot\,,\cdot\,],\w)$ be a nonzero symplectically irreducible symplectic Lie algebra. Then, $[\g,\g]$ is an orthogonal direct sum of $2$-dimensional ideals of $(\g,\cro,\w)$. Moreover, each ideal of this decomposition containing no $1$-dimensional ideal of $(\g,\cro)$.
\end{corollary}

\begin{proof}
	Let us consider the representation $\pi$ of $[\g,\g]^\perp$ on $[\g,\g]$ defined in Theorem \ref{T1SymplIrredu}. By Lie's theorem (cf. \cite[p. 46]{Bour}), there exists either a $1$-dimensional vector subspace $\ii$ of $[\g,\g]$ stable under $\pi$ or a $2$-dimensional vector subspace $\ii$ of $[\g,\g]$ stable under $\pi$ without any $1$-dimensional vector subspaces stable under $\pi$. It is clear that $\ii$ is an ideal of $(\g,\cro)$, and since $(\g,\cro,\w)$ is symplectically irreducible, the dimension of $\ii$ must be equal to $2$. In addition, $\ii$ is nondegenerate and therefore $[\g,\g]= \ii\oplus(\ii^\perp\cap[\g,\g])$. Since the action of $[\g,\g]^\perp$ on $[\g,\g]$ is skew-symmetric, it follows that $\ii^\perp\cap [\g,\g]$ is stable under $\pi$, and consequently it is a nondegenerate ideal of $(\g,\cro)$. Thus, $\ii^\perp\cap [\g,\g]$ also decomposes into the orthogonal direct sum of a $2$-dimensional ideal, containing no $1$-dimensional ideal of $(\g,\cro)$, and its orthogonal which is also a nondegenerate ideal of $(\g,\cro,\w)$. Continuing this process, we obtain that  $[\g,\g]$ is an orthogonal direct sum of $2$-dimensional ideals of $(\g,\cro,\w)$ and each ideal of this decomposition contains no $1$-dimensional ideal of $(\g,\cro)$.
\end{proof}

\begin{remark}
Note that the ideals obtained in Corollary \ref{CompCharaOfIrrSymLiAlg0} are mutually non-isomorphic (see  Corollary \ref{CompCharaOfIrrSymLiAlg} or \cite[p. 39]{Bau}).
\end{remark}

The next corollary provides a complete characterization of symplectically irreducible symplectic Lie algebras.

\begin{corollary}\label{CompCharaOfIrrSymLiAlg}
	Let $(\g,[\,\cdot\,,\cdot\,],\w)$ be a nonzero symplectically irreducible symplectic Lie algebra. Then, there exists a basis $\mathbb{B}:=\{e_1,\dots,e_{2s},x_1,y_1\dots,x_m,y_m\}$ of $\g$ with $m\geqslant 2s$ such that:
	\begin{enumerate}
		\item $[\g,\g]^\perp=\Span_\rel\{e_1,\dots,e_{2s}\}$, and $\,[\g,\g]=\Span_\rel\{x_1,y_1,\ldots,x_m,y_m\}$;
		\item There exist $\widetilde{e}_1,\ldots,\widetilde{e}_m\in[\g,\g]^\perp$ such that:
		\begin{equation*}
			[e_i,x_k]= -\w(\widetilde{e}_k,e_i)\,y_k,\et [e_i,y_k]= \w(\widetilde{e}_k,e_i)\,x_k,
		\end{equation*}  
		for all $i\in \{1,\dots,2s\}$, and $k\in \{1,\dots,m\}$; 
		\item The symplectic form $\w$ can be expressed in the basis $\mathbb{B}$ as follows:
		\begin{equation*}
			\omega=\sum_{i=1}^{s}\,e_i^*\wedge e_{s+i}^*+\sum_{k=1}^{m}\epsilon_k\,x_k^*\wedge y_k^*,
		\end{equation*}
		where $\epsilon_1,\ldots,\epsilon_m\in{\Bbb R}\backslash \{0\}$;
		\item For any $i,j\in\{1,\dots,2s\}$, we have:
		\begin{equation*}
			\kappa_\g(e_i,e_j)=\sum_{k=1}^{m}\w(\widetilde{e}_k,e_i)\w(e_j,\widetilde{e}_k).
		\end{equation*}
	\end{enumerate}
\end{corollary}

\begin{proof}
	From Corollary \ref{CompCharaOfIrrSymLiAlg0}, we have $[\g,\g]=\ii_1\oplus\cdots\oplus\ii_m$, where each $\ii_k$ is a $2$-dimensional ideal of $(\g,\cro)$, irreducible under the action of $[\g,\g]^\perp$, and contains no $1$-dimensional ideal of $(\g,\cro$). Thus, by Lemma \ref{LemSolIrrDim2}, for each $k\in\{1,\ldots,m\}$ there exists a basis $\{x_k,y_k\}$ of $\ii_k$ and $\alpha_k,\beta_k\in\left([\g,\g]^\perp\right)^*$ such that:
	\begin{equation*}
		\left\{
		\begin{array}{r c l}\vspace{0.5mm}
			\left[a,x_k\right]&=&\alpha_k(a)x_k-\beta_k(a)y_k,\\ 
			\left[a,y_k\right]&=&\beta_k(a)x_k+\alpha_k(a)y_k,
		\end{array}
		\right.\\
	\end{equation*}
	for any $a\in[\g,\g]^\perp$. Moreover, the fact that the action of $[\g,\g]^\perp$ on $\ii_k$ is skew-symmetric implies that $\alpha_k=0$. Hence,
	\begin{equation*}
		[a,x_k]= -\beta_k(a)\,y_k,\et [a,y_k]= \beta_k(a)\,x_k.
	\end{equation*}  
	Now, since $\big([\g,\g]^\perp,\w_{|[\g,\g]^\perp\times [\g,\g]^\perp}\big)$ is nondegenerate, we can choose a symplectic basis $\{e_1,\ldots,e_{2s}\}$ that satisfies all the properties mentioned in the corollary. Furthermore, for each $\beta_k\in\left([\g,\g]^\perp\right)^*$, there exists a unique $\widetilde{e}_k\in[\g,\g]^\perp$ such that $\beta_k=\w(\widetilde{e}_k,\cdot\,)$. Finally, the condition that $m=\tfrac{1}{2}\dim[\g,\g]$ must necessarily be greater than or equal to $2s=\dim[\g,\g]^\perp$ follows from the facts that the action of $[\g,\g]^\perp$ on $[\g,\g]$ is faithful, and for any $a\in [\g,\g]^\perp$, the matrix representation of $\pi(a)$ in the basis $\{x_1,y_1\dots,x_m,y_m\}$ is given by:
	\begin{equation*}
		\begin{pmatrix}
			\begin{pmatrix}
				0&\w(\widetilde{e}_1,a)\\
				-\w(\widetilde{e}_1,a)&0
			\end{pmatrix}&&\text{\Huge0}\\
			&\ddots&\\
			\text{\Huge0}&&\begin{pmatrix}
				0&\w(\widetilde{e}_m,a)\\
				-\w(\widetilde{e}_m,a)&0
			\end{pmatrix}
		\end{pmatrix}.
	\end{equation*}
\end{proof}

\begin{corollary}\label{CorNotExisDimLes6}
	There is no nonzero symplectically irreducible symplectic Lie algebra of dimension less than $6$.
\end{corollary}

It was noted in \cite[p. 778]{DM2}, without proof, that the Lie algebra underlying any symplectically irreducible $6$-dimensional symplectic Lie algebra is isomorphic to that given in Example \ref{ExSympIrrDim6}. The following proposition provides a rigorous proof of this statement.

\begin{proposition}\label{SymplecIrreSympLiAlgebOfDim6}
	Let $(\g,[\,\cdot\,,\cdot\,],\w)$ be a symplectically irreducible $6$-dimensional symplectic Lie algebra. Then, there exists a basis $\mathbb{B}_0:=\{\hat{e}_k\}_{1\leqslant k\leqslant 6}$ of $\g$ such that
	\begin{equation}\label{LiBraSympIrrDim6}
		[\hat{e}_1,\hat{e}_3]= \hat{e}_4,\quad [\hat{e}_1,\hat{e}_4]=-\hat{e}_3,\quad [\hat{e}_2,\hat{e}_5]= \hat{e}_6,\et [\hat{e}_2,\hat{e}_6]=-\hat{e}_5,
	\end{equation} 
	are the only nonzero Lie brackets of $(\g,\cro)$.
\end{proposition} 

\begin{proof}
	According to Corollary \ref{CompCharaOfIrrSymLiAlg}, there exists a basis $\mathbb{B}:=\{e_1,e_2,x_1,y_1,x_2,y_2\}$ of $\g$ such that:
	\begin{eqnarray*}
		&[e_1,x_1]=-\xi_1 y_1,\quad [e_1,y_1]=\xi_1x_1,\quad [e_1,x_2]=-\xi_2y_2,\quad [e_1,y_2]=\xi_2x_2,\\
		&[e_2,x_1]=-\xi_3 y_1,\quad [e_2,y_1]=\xi_3x_1,\quad [e_2,x_2]=-\xi_4y_2,\quad [e_2,y_2]=\xi_4x_2,
	\end{eqnarray*}
	where $\xi_k\in\rel$ for $k\in\{1,\ldots,4\}$. We claim that $\xi_1\xi_4\neq\xi_2\xi_3$. Indeed, assuming that $\xi_1\xi_4=\xi_2\xi_3$, we consider the following two cases:
	\begin{itemize}
		\item[$1.$] If $\xi_1=0$, then either $\xi_2=0$ or $\xi_3=0$. However, the fact that $\xi_1=\xi_2=0$ (resp. $\xi_1=\xi_3=0$) would implies that $\rel e_1$ (resp. $\rel x_1$) is an isotropic ideal of $(\g,\cro,\w)$, which is impossible.
		\item[$2.$] If $\xi_1\neq 0$, then using the fact that $\xi_4=\tfrac{\xi_2\xi_3}{\xi_1}$, one can easily verify that $\rel\widetilde{e}$ is an isotropic ideal of $(\g,\cro,\w)$, where $\widetilde{e}:=e_2-\tfrac{\xi_3}{\xi_1}e_1$. But this contradicts the fact that $(\g,\cro,\w)$ is symplectically irreducible.
	\end{itemize}
	In summary, $\xi:=\xi_1\xi_4-\xi_2\xi_3\neq 0$. Now define:
	\begin{equation*}
		\hat{e}_1:=\tfrac{\xi_4}{\xi}e_1-\tfrac{\xi_2}{\xi}e_2,\quad\hat{e}_2:=\tfrac{\xi_3}{\xi}e_1-\tfrac{\xi_1}{\xi}e_2,\quad\hat{e}_3:=y_1,\quad\hat{e}_4:=x_1,\quad\hat{e}_5:=x_2,\et\hat{e}_6:=y_2.
	\end{equation*}
	A straightforward computation shows that the Lie bracket of $\g$ in the basis $\mathbb{B}_0:=\{\hat{e}_k\}_{1\leqslant k\leqslant 6}$ are given by \eqref{LiBraSympIrrDim6}.
\end{proof}

We conclude this paper by presenting an example of an $8$-dimensional solvable symplectic Lie algebra that can be obtained as a (normal) symplectic double extension of the $6$-dimensional symplectically irreducible symplectic Lie algebra by a line. Moreover, we show that there is no isotropic ideal with dimension $1$ or $2$ of this $8$-dimensional solvable symplectic Lie algebra that would allow its construction, via the symplectic double extension process, starting from the zero symplectic Lie algebra.

\begin{example}
For $\mu\in\rel$, consider the $8$-dimensional Lie algebra $(\g_\mu,\cro_\mu)$ defined by:
\begin{eqnarray*}
&&[e_1,e_3]_\mu=e_4-e_8,\quad [e_1,e_4]_\mu=-e_3+e_8,\quad [e_1,e_7]_\mu=-e_3-e_4,\quad [e_2,e_5]_\mu=e_6,\quad[e_2,e_6]_\mu=-e_5,\\
&&[e_3,e_7]_\mu=-e_8,\quad [e_4,e_7]_\mu=-e_8,\quad [e_5,e_7]_\mu=\mu e_6,\quad
[e_6,e_7]_\mu=-\mu e_5,\et [e_7,e_8]_\mu=e_8,
\end{eqnarray*}
where $\mathbb{B}:=\{e_k\}_{1\leqslant k\leqslant 8}$ is a basis of $\g_\mu$, and the unspecified Lie brackets are either zero or given by skew-symmetry. Consider the following nondegenerate skew-symmetric bilinear form on $\g_\mu$ defined by:
\begin{equation*}
\w:=e_1^*\wedge e_2^*+e_3^*\wedge e_4^*+e_5^*\wedge e_6^*+e_7^*\wedge e_8^*.
\end{equation*}
A direct computation shows that $\w$ is a symplectic form on $(\g_\mu,\cro_\mu)$. Moreover, put $e_0:=e_8$, then it is clear that $\mathfrak{i}:=\rel e_0$ is a $1$-dimensional normal isotropic ideal of $(\g_\mu,\cro_\mu,\w)$. Thus, by Theorem \ref{TheoDoubExtdim1}, $(\g_\mu,\cro_\mu,\w)$ is a symplectic double extension by a line of a $6$-dimensional symplectic Lie algebra $(\ai,\cro_{\ai},\w_{\ai})$. Let us find the admissible element $(a^{\ai},b^{\ai},\D,\varepsilon)$. Put $d_0:=e_7$, $a_k:=e_k$ for $k\in\{1,\ldots 6\}$, and $\ai=\Span_\rel\{a_1,a_2,a_3,a_4,a_5,a_6\}$. We endow the vector space $\ai$ with the following symplectic Lie algebra structure:
\begin{eqnarray*}
	&&[a_1,a_3]_{\ai}= a_4,\quad [a_1,a_4]_{\ai}=-a_3,\quad [a_2,a_5]_{\ai}= a_6,\quad [a_2,a_6]_{\ai}=-a_5,\\
	&&\et \w_{\ai}:=a^*_1\wedge a^*_2+a^*_3\wedge a^*_4+a^*_5\wedge a^*_6.
\end{eqnarray*} 
Consider the following linear endomorphism $\D_\mu\in\gl(\ai)$ defined by:
\begin{equation*}
\D_\mu:\ai\longrightarrow\ai,\qquad \D_\mu\big(\xi_1a_1+\cdots+\xi_6a_6\big):=\xi_1a_3+\xi_1a_4+\xi_6\mu a_5-\xi_5\mu a_6.
\end{equation*}
A small computation yields that $\D_\mu\in\mathfrak{der}(\ai)$, and that:
\begin{equation*}
\varphi_{\D_\mu}(a,b):=\w_{\ai}\big(\D_\mu(a),b\big)+\w_{\ai}\big(a,\D_\mu(b)\big)=\delta\big(i_{a_3-a_4}\w_{\ai}\big)(a,b),\qquad\forall\, a,b\in\ai_\mu.
\end{equation*}
As a result, $(a^{\ai}:=0,b^{\ai}:=a_3-a_4,\D:=\D_\mu,\varepsilon:=-1)$ is an admissible element of the $6$-dimensional symplectically irreducible symplectic Lie algebra  $(\ai,\cro_{\ai},\w_{\ai})$, and $(\g_\mu,\cro_\mu,\w)$ is its symplectic double extension by a line. We now verify that this is the only possible way to obtain the $8$-dimensional symplectic Lie algebra $(\g_\mu,\cro_\mu,\w)$ using our symplectic double extension process. Specifically, we show that $\rel e_8$ is the only nonzero isotropic minimal ideal of $(\g_\mu, \cro_\mu, \w)$. Let $\jj\subset\g_\mu$ be a nonzero minimal ideal of $(\g_\mu,\cro_\mu)$ and $x=\sum_{k=1}^8x_ke_k\in\jj\backslash\{0\}$. We have:
\begin{eqnarray}\label{LiBraExample8SymDoubExtofIrr}
\left[e_1,x\right]&=&-(x_4+x_7)e_3+(x_3-x_7)e_4+(x_4-x_3)e_8;\nonumber\\
\left[e_2,x\right]&=&-x_6e_5+x_5e_6;\nonumber\\
\left[e_3,x\right]&=&-x_1e_4+(x_1-x_7)e_8;\nonumber\\
\left[e_4,x\right]&=&x_1e_3-(x_1+x_7)e_8;\nonumber\\
\left[e_5,x\right]&=&(\mu x_7-x_2)e_6;\\
\left[e_6,x\right]&=&(x_2-\mu x_7)e_5;\nonumber\\
\left[e_7,x\right]&=&x_1e_3+x_1e_4+\mu x_6e_5-\mu x_5e_6+(x_3+x_4+x_8)e_8;\nonumber\\
\left[e_8,x\right]&=&-x_7e_8.\nonumber
\end{eqnarray}  
If $x_7\neq 0$, then by minimality $\jj=\rel e_8$. Suppose that $x_7=0$, then either one of $x_1$ and $x_2$ is nonzero or they are both equal to zero. If $x_1\neq 0$ (resp. $x_2\neq 0$), then $\jj=\Span_\rel\{e_3-e_8,e_4-e_8\}$ (resp. $\jj=\Span_\rel\{e_5,e_6\}$), and in this case $\jj$ is nondegenerate. If $x_1=x_2=0$, then $x=\sum_{k=3}^6x_ke_k+x_8e_8$, and \eqref{LiBraExample8SymDoubExtofIrr} becomes:
\begin{eqnarray}\label{LiBraExample8SymDoubExtofIrr-2}
\left[e_1,x\right]&=&-x_4e_3+x_3e_4+(x_4-x_3)e_8;\nonumber\\
	\left[e_2,x\right]&=&-x_6e_5+x_5e_6;\\
	\left[e_7,x\right]&=&\mu x_6e_5-\mu x_5e_6+(x_3+x_4+x_8)e_8.\nonumber
\end{eqnarray} 
If $x_8\neq -(x_3+x_4)$, then $\mu[e_2,x]+[e_7,x]\in\jj\backslash\{0\}$, and this implies, by minimality, that $\jj=\rel e_8$. On the other hand, if $x_8=-(x_3+x_4)$, then $x=\sum_{k=3}^6-(x_3+x_4)e_8$, and \eqref{LiBraExample8SymDoubExtofIrr-2} becomes:
\begin{eqnarray}\label{LiBraExample8SymDoubExtofIrr-3}
	\left[e_1,x\right]&=&-x_4e_3+x_3e_4+(x_4-x_3)e_8;\nonumber\\
	\left[e_2,x\right]&=&-x_6e_5+x_5e_6;\\
	\left[e_7,x\right]&=&\mu x_6e_5-\mu x_5e_6.\nonumber
\end{eqnarray}
If $x_5\neq 0$ or $x_6\neq 0$, then $[x,e_2]\in\jj\backslash\{0\}$, and hence, by minimality, $\jj=\Span_\rel\{e_5,e_6\}$. If $x_5=x_6=0$, then $x=x_3e_3+x_4e_4-(x_3+x_4)e_8$. Consequently, either $x_3\neq 0$, or $x_4\neq 0$, since otherwise $x$ will be equal to zero, which is impossible. Thus,
\begin{equation*}
x_4(e_3-e_8)-x_3(e_4-e_8)=[x,e_1]\in\jj\backslash\{0\}.
\end{equation*}
It follows that $\jj=\Span_\rel\{e_3-e_8,e_4-e_8\}$, which is nondegenerate. In summary, $\rel e_8$ is the only isotropic minimal ideal of $(\g_\mu,\cro_\mu,\w)$. 
\end{example}

\begin{remark}
Note that, $\g_\mu=\jj^\perp\ltimes\jj$, where $\jj:=\Span_\rel\{e_5,e_6\}$ is a symplectic abelian ideal of $(\g_\mu,\cro_\mu,\w)$, and $\jj^\perp=\Span_\rel\{e_1,e_2,e_3,e_4,e_7,e_8\}$ is a symplectic Lie subalgebra of it. Furthermore, since $\dim[\jj^\perp,\jj^\perp]=3$, $\jj^\perp$ is not symplectically irreducible. Consequently, both $\jj$ and $\jj^\perp$ can be obtained through a finite sequence of symplectic double extensions by a line, starting from the zero symplectic Lie algebra.
\end{remark}

\section{Concluding remarks} The main conclusion of this paper is that any solvable, in particular unimodular, symplectic Lie algebra is either symplectically irreducible or can be obtained through a finite sequence of symplectic double extensions by a line, a plane, or both, starting from a symplectically irreducible one.

In \cite{Fi}, M. Fischer focused exclusively on the case where the center is nonzero and degenerate, particularly in the nilpotent (non-abelian) setting. The main objectives of Fischer’s work were to compare symplectic double extensions constructed from different choices of $1$-dimensional central ideals, which are naturally isotropic and normal, and to address the possibility of obtaining isomorphic symplectic double extensions from non-isomorphic symplectic reductions. In future work, we intend to investigate these problems in our own framework. This study promises to be particularly interesting—though more involved—as it will require examining extensions beyond semi-direct products and central extensions by $1$-dimensional Lie algebras.

\bibliographystyle{amsplain}

\end{document}